\documentclass[12pt,a4paper]{amsart}
\usepackage{hyperref}
\usepackage{comment}

\makeatletter
\renewcommand\normalsize{%
    \@setfontsize\normalsize{11.7}{14pt plus .3pt minus .3pt}%
    \abovedisplayskip 10\p@ \@plus4\p@ \@minus4\p@
    \abovedisplayshortskip 6\p@ \@plus2\p@
    \belowdisplayshortskip 6\p@ \@plus2\p@
    \belowdisplayskip \abovedisplayskip}
\renewcommand\small{%
    \@setfontsize\small{9.5}{12\p@ plus .2\p@ minus .2\p@}%
    \abovedisplayskip 8.5\p@ \@plus4\p@ \@minus1\p@
    \belowdisplayskip \abovedisplayskip
    \abovedisplayshortskip \abovedisplayskip
    \belowdisplayshortskip \abovedisplayskip}
\renewcommand\footnotesize{%
    \@setfontsize\footnotesize{8.5}{9.25\p@ plus .1pt minus .1pt}%%
    \abovedisplayskip 6\p@ \@plus4\p@ \@minus1\p@
    \belowdisplayskip \abovedisplayskip	
    \abovedisplayshortskip \abovedisplayskip
    \belowdisplayshortskip \abovedisplayskip}
\ifdefined\pdfpagewidth
\else
\fi
\calclayout
\makeatother

\usepackage{amssymb,amscd,amsthm,mathrsfs,enumitem,color,transparent,xcolor}
\usepackage[utf8]{inputenc}
\usepackage{epsfig,graphicx,psfrag,mathrsfs,float}  %figuras  \usepackage{wrapfig,subfigure}
\usepackage{tikz}
\usepackage{mathtools}
\usepackage{array}

\DeclarePairedDelimiter\abs{\lvert}{\rvert}%

\newcolumntype{C}[1]{>{\centering\let\newline\\\arraybackslash\hspace{0pt}}m{#1}}

\newtheorem{mainthm}{Theorem}
\newtheorem{maincor}{Corollary}
\newtheorem*{thm*}{Theorem}
\newtheorem*{prop*}{Proposition}

\newtheorem*{corB}{Corollary B}

\newtheorem{thm}{Theorem}[section]

\newtheorem{lemma}[thm]{Lemma}
\newtheorem{lema}[thm]{Lemma}
\newtheorem{prop}[thm]{Proposition}
\newtheorem{cor}[thm]{Corollary}

\newcommand{\bi}{\begin{itemize}}
\newcommand{\ei}{\end{itemize}}

\theoremstyle{definition}

\theoremstyle{remark}
\newtheorem{remark}[thm]{Remark}
\newtheorem{obs}[thm]{Remark}

\newcommand{\T}{\mathbb{T}}
\newcommand{\R}{\mathbb{R}}
\newcommand{\Z}{\mathbb{Z}}
\newcommand{\N}{\mathbb{N}}

\newcommand{\A}{\mathbb{A}}

\title[Chaos and Diffusion
on the annulus]
{Conditions implying Annular Chaos: Quantitative results and Computer-Assisted Proofs}

\author{M. J. Capi\'nski}
\address[Capi\'nski]{Faculty of Applied Mathematics, AGH University of Krak\'ow, Kraków, Poland}
\email{maciej.capinski@agh.edu.pl}

\author{M. Gr\"oger}
\address[Gr\"oger]{Faculty of Mathematics and Computer Science, Jagellonian University, Kraków, Poland}
\email{maik.groeger@im.uj.edu.pl}

\author{A. Passeggi}
\address[Passeggi]{CMAT, Facultad de Ciencias, Universidad de la Rep\'{u}blica,
Igu\'{a} 4225, 11400 Montevideo, Uruguay} \email{apasseggi@cmat.edu.uy}

\author{F.A. Tal}
\address[Tal]{Instituto de Matem\' atica e Estat\' istica da Universidade de S\~ao Paulo,
R. do Mat\~ ao, 1010 - Vila Universitaria, S\~ ao Paulo, Brasil}
\email{fabiotal@ime.usp.br}

\thanks{M.C. was partially supported by the NCN grant 2021/41/B/ST1/00407. }

\begin{document}

\maketitle
 
\begin{abstract}

We derive quantitative sufficient conditions for rotational chaos and diffusion in annular homeomorphisms, building on the topological criteria established in \cite{paper1}. These conditions depend only on basic properties of the maps, making their implementation straightforward. To demonstrate the effectiveness of the method, we provide computer-assisted proofs of rotational chaos and diffusion for classical families of annular maps and their variations, in both conservative (twist and non-twist) and dissipative settings.

\end{abstract}

%\tableofcontents

\section{Introduction}

Despite significant progress in understanding the concept of \emph{chaos} since the early 20th century, it remains mathematically challenging to determine whether chaotic behavior occurs in a given prescribed system. Even the foundational examples that sparked the development of the theory (the \emph{Three-Body Problem} and the \emph{Van der Pol Equation}) are still not fully understood. From the geometric perspectives of Poincaré and Birkhoff, current approaches continue to fall short of providing an ideal framework for applications \cite{poincaregeometric,birk4}.
Beyond basic models and well-established theoretical criteria, no systematic method exists for detecting chaos that does not rely heavily on the specifics of the system under consideration. In fact, key theoretical concepts needed for a systematic approach—such as \emph{instability regions}, \emph{Birkhoff attractors}, and \emph{homoclinic points}—are extremely difficult to apply without subtle and detailed information, particularly in parameter families that include integrable cases.

\smallskip
Recently, topological criteria have been obtained for the existence of annular chaos, i.e., the presence of rotational horseshoes, which depend on basic properties of maps on the annulus and can therefore be easily implemented \cite{paper1}. These results translate classical concepts introduced by Birkhoff (such as instability regions in the Hamiltonian case and Birkhoff attractors in the dissipative case) into a set of conditions that are satisfied once a finite orbit connecting two prescribed open sets $A$ and $B$ is found.
These open sets can be readily computed once the map is given (see Section \ref{s.prev}).

\smallskip

This article builds upon the mentioned criteria, extending the results of \cite{paper1} into a metric framework. Our new theoretical contributions advance the field of Surface Dynamics, particularly the rotation theory of the two-torus both for homeomorphisms in the isotopy class of the identity (see \cite{MiZi,frankstoro,llibreMcKay,kwapiszratpol,kwapisznonpol,korotal}
as a sample of basic literature), as well as for homeomorphisms homotopic to Dehn twists (see \cite{doeff1997rotation,addas2002existence}). One of the consequences of the new results permits computable bounds on the possible deviations with respect to a rotation interval for area-preserving maps, see Corollary \ref{c.coruno}. 
This addresses an important version of a longstanding open problem in the field, dating back to the first results demonstrating the existence of theoretical bounds for deviations with respect to certain rotation sets, but where actual bounds were not usually known, except for homeomorphisms in the Dehn twist isotopy class \cite{addas2012dynamics}, see
Corollary \ref{c.coruno} and the references immediately preceding it. Note that for classical twist diffeomorphims bounds on vertical displacement can also be obtained from the vertical diameter of invariant curves, and the diffusion is equivalent to the non-existence of any invariant curve. 

A second aspect of the theoretical results introduced here (arguably the most compelling) is the development of clear and practical methods for detecting chaos in prescribed maps that do not rely on subtle or map-specific properties. In this work, we present a first demonstration of how these methods operate in practice, using well-known families of maps as testing grounds and without requiring a deep, map-specific analysis. Specifically, we provide Computer-Assisted Proofs (CAPs) of the existence of diffusion and chaos across various elements of these families, covering a wide range of parameters and variations. These include the Standard Family (SF), Variations of the Standard Family (VSF), the Non-Twist Standard Family (NTSF), and the dissipative version of the Standard Family (DSF). 
All these CAPs follow a common strategy: in theory, one seeks an orbit connecting two given open sets, $A$ and $B$; such an orbit is then located using computer-assisted techniques. To ensure that the numerically obtained orbit corresponds to a true orbit, numerical errors must be rigorously controlled, which is achieved through the use of \emph{Interval Arithmetic}. Notably, the map's differential plays no essential role in any of these implementations, marking a clear contrast with prior approaches to the study of chaos. We now state the theoretical results of this work.

\subsection{Theoretical results}

Let us introduce the theoretical results of this work for which some notation is needed. In this article we are interested in quantitative results, so we work in a metric annuli. The model
we choose is given by $\R^2/_\sim$, where $x\sim y$ if their first coordinate differ by an integer. 
We denote by $\textrm{Homeo}_{\,0}(\A)$ the space of homeomorphisms of $\A$ that are isotopic to the identity. 
Then we will be interested in some subfamilies. Define $\T^2$  as $\A/_{\sim}$ with $x\sim y$ whenever the second
coordinates of representatives of lifts of $x,y$ differ in integer values, and consider the following spaces.

\begin{enumerate}

\item $\mathrm{Homeo}_{\,Dehn\,k,nw}(\A)$ given by those lifts of non-wandering homeomorphisms of $\T^2$ which
are in the homotopic class of the projection (in $\T^2$) of the planar map $D_k(x,y)=(x+k\,y\,,\,y)$ for $k\neq 0$, that is, a Dehn twist.

\item $\mathrm{Homeo}_{\,0,\,nw}(\A)$ given by those lifts of homotopic to the identity and non-wandering homeomorphisms of $\T^2$.

\item $\mathrm{Homeo}_{\,0,\,nw,\rho}(\A)$ given by those elements $\hat{f}$ of 
$\mathrm{Homeo}_{\,0,nw}(\A)$ having fixed points with a rotational difference 
$\rho$ (the definition of this standard concept can be found in Section~\ref{s.prev}).

\end{enumerate}

An element of any of these families lifting a map on $\T^2$ which preserves an absolutely continuous measure 
$\lambda$ with respect to Lebesgue is said to have \emph{zero vertical drift} whenever
$$\int_{\T^2} \phi(x)\,d\lambda=0,$$
where $\phi(x):=\textrm{pr}_2(f(y)-y)$ for any lift $y\in\A$ of $x\in\T^2$.
\smallskip

Let us state the basic dynamical data that we want to find here.
For a map in any of such families, we say that it has \emph{unbounded diffusion} whenever we find orbits 
in $\A$ with vertical displacement arbitrary close to $+\infty$ or arbitrary close to $-\infty$. When both
conditions holds we say that $f$ has \emph{total unbounded diffusion}.
The other definition we need is that of \emph{rotational chaos}. An annular homeomorphism $f$ homotopic
to the identity has a rotational horseshoe whenever the image $f^n(R)$ of some topological rectangle $R$
intersects $R$ in at least two Markovian components, so that these cannot be lifted to the same lift of $R$. 
Such intersections produce a semi-conjugation to the two symbol shift, where each symbol gives
an extra information about displacements of orbits in the universal covering. These two concepts have been
object of deep studies in both the rigorous treatment of chaos and also the numerical one, both in mathematics
and physics, and our theoretical results conform a set of sufficient conditions for finding the just 
introduced properties for annular maps. It is interesting to note that in those cases where total diffusion and
the existence of a rotational horseshoe is proven, the \emph{chaotic sea} obtained by the unstable lamination
of the horeseshoe will be unbounded in both directions of $\A$ and will absorve the topology of the space.

\smallskip

Given a map $\hat{f}\in\mathrm{Homeo}_0(\A)$ define its \emph{maximal vertical displacement} as
$$N(f)=\sup_{x\in\A}\{\mathrm{pr}_2(\hat{f}(x)-x)\}$$
where $\textrm{pr}_2$ is the projection on the second coordinate naturally defined on $\A$. 
We have the following results.

\begin{mainthm}\label{thm:thm1}

Let $f$ belong to  $\mathrm{Homeo}_{\,Dehn\,k\,,\,nw}(\A)$, $k\geq 1$ and let
$$ M=\min\{3\,N(f)+2\,,\,N(f)+4\}. $$
Then, the existence of $z\in\A\,,n\in\Z$ with
$$|\mathrm{pr}_2(f^n(z)-z)|\geq M$$
implies that $f$ has unbounded diffusion. Moreover, if $f$ has a fixed point or zero vertical drift, then $f$ has rotational chaos and total unbounded diffusion.

\end{mainthm}

For the non-twist case, we have the following version of the result. Before stating it, we should recall
that there is a largely developed theory concerning rotation sets of toral homeomorphisms in the homotopy class
of the identity, where the seminal result
is that this set is given by a compact and convex set \cite{MiZi} of $\R^2$. An important result is that in case the interior of the set is non-empty, then we have positive entropy associated to rotational horseshoes \cite{llibreMcKay}.

\begin{mainthm}\label{thm:thm2}

Let $f$ belong to $\mathrm{Homeo}_{\,0,\,nw,\rho}(\A)$, with a rotational difference $\rho\geq 1$, and let 
$$M_1= 6N(f)+2\ ,\ M_2=4N(f)+2\mbox{ and }M_i =2N(f)+2\mbox{ for }i\ge 3.$$ 
Then, the existence of $z\in \A,\,n\in\Z$ such that 
$$\abs{\mathrm{pr}_2\left(f^n(z)-z\right)}\ge M_{\rho}$$ implies that
$f$ has unbounded diffusion, and in particular $f$ projects to 
an homeomorphism $\check f$ of $\T^2$ whose associated rotation set has non-empty interior.
Furthermore, if $f$ has zero vertical drift, it has rotational chaos and total unbounded diffusion.
\end{mainthm}

This result turns out to be quite significant for the rotation theory on the two-torus 
$\T^2$, not only for the direct conclusion about the \emph{non-empty interior} of rotation sets, but also for those cases where the rotation set is a segment containing two rational points, a setting in which the problem of \emph{deviations} with respect to rotation sets is one of the main challenges. Until now, the state of the art claimed that the deviations with respect to a rotation set given by a segment containing two rational points must be bounded, but there was no way to obtain a computable bound from simple properties of the map, see for instance \cite{davalos,davalosseg,nakotal,korotaldos,korotal,zanata2}. 
The following corollary resolves this.

\begin{maincor}\label{c.coruno}
Let $f:\T^2\to\T^2$ be a non-wandering homeomorphism homotopic to the identity, whose rotation set is an horizontal interval containing a pair of integers whose difference is $\rho\geq 1$. Then the deviations for $f$ is bounded by $M_{\rho}$ as defined
above.
\end{maincor}

\begin{obs}
The hypothesis that the rotation set is a horizontal segment does not impose major restrictions, since for any map whose rotation set is an interval containing two rational points, some power of the map is conjugate to one satisfying the hypotheses of the corollary, see \cite{kwapiszratpol}.
\end{obs}

As mentioned before, Theorem \ref{thm:thm1} and \ref{thm:thm2} will be applied to different families below, giving CAPs of the existence of rich dynamics
in a wide and systematic manner. Although our cases are mostly analytical, we notice that one could
work with $C^0$ maps without implying any extra difficulty: we do not require any information about the differential of
the functions, which shows a huge difference with the previously involved methods for proving chaos in given families. 
This fact is crucial for our purposes, as we expect to work with Poincaré return maps in the near future. Given the difficulty of controlling their derivatives, these maps are considered in the $C^0$ topology.

\smallskip 
Another property that is frequently used in the literature is the so-called
\emph{twist condition}.
This condition together with that of being analytical serves a set of tools which have
been largely developed in classic treatment of chaos for annular maps, in both the theoretical and numerical treatment 
(see for instance \cite{twistregimMackay,starktwist}).
The applications of Theorem \ref{thm:thm1} and Theorem \ref{thm:thm2} done for variations of the so called Standard Family show
 that we do not require this condition at all. On the other hand, in the next result we combine the basic tools for the twist regime the new input obtained here, and deduce an important result for the so-called
Non-Twist Standard Family (NTSF) \cite{NTSFscholarpedia}. These maps
are given by those having as lifts in $\R^2$
%$$ f_{a,b}(x,y)= (x+(y+a \sin(2\pi x))^2+b \,,\, y+a \sin(2\pi x)),$$  OLD VERSION / OLD PROOF COMMENTED 
\begin{equation}
	f_{a,b}(x,y)=\left(x+a\,\left[ 1- ( y-b\,\sin(2\pi\,x))^2\right],\,  y-b\,\sin(2\pi\,x)\right). \label{eq:NTSF-formula}
\end{equation}

As before, we abuse notation by referring to the induced maps on $\A$ using the same names. This family of maps constitute the non-twist model for annular chaos, which has gained increasing importance in recent years as the foundation for various physical models \cite{NTSFscholarpedia,ntsmmemoirs,nontwistphysicad,nontwist1}. 
Since this family lies outside the twist regime near $y = 0$, classical techniques previously used in the SF setting cannot be globally applied unless one restricts to regions far from $y = 0$. In contrast, our techniques remain effective in this critical region. By combining both approaches we obtain point (2) of the following results. For point (1) we make use of a combination of results, some being recent \cite{lecalveztal,Koromeysam,jager2021onset}. It is important to note while reading the proof of this result, that what is done for proving (1) works for any analytic family of annular maps. 

\begin{mainthm}\label{thm:thm3}
Let $a \in (0,\infty)$ and $b\in\R$ be given, and set 
%$$M_a= \max\left\{ 3\,a +\frac{2}{a}\, ,\, 2\,a +\frac{4}{a}\right\}$$
\begin{equation}
	M_{a,b}=\,\max\left\{  2\,\sqrt{1+\frac{3}{a}}+|b|\,,\,\frac{2}{a\,|b|}\right\}. \label{eq:Mab-ntsf}
\end{equation}
Then,
\begin{enumerate}

\item For all parameters $a \in (0,\infty)$ and $b\in\R\setminus\{0\}$ the map $f_{a,b}$ has a rotational horseshoe.

\item If there exists a point $z\in\A$ with $\mathrm{pr}_2(z)<-M_{a,b}$ and an integer $n$ such that $\mathrm{pr}_2(f^n_{a,b}(z))>M_{a,b}$, then $f_{a,b}$ has total unbounded diffusion.
\end{enumerate}
\end{mainthm}

Next, we present the results of CAPs demonstrating the existence of rich dynamics for various families and parameters. 
Again, we believe the key comparison with the previous state of the art lies not in the detailed description of a specific map, 
but in how easily the criteria can be applied—requiring no deep information about the map.

\subsection{Twist and non-twist variations of the Standard Family}

Based on Theorem \ref{thm:thm1} and \ref{thm:thm2} we obtain CAPs for the existence of diffusion
and rotational horseshoes. 
This, in particular, will show how easy it is to apply the obtained criteria from one case to another.
Namely, the criteria do not depend on subtle properties of the given maps and still
cover a substantial range of situations once compared with the previous results in the field. In order to show this, 
we introduce CAPs for the existence of unbounded diffusion and rotational chaos for the following family, obtained
by variations of the Standard Family in the twist and non-twist regime.

The Standard Family (SF) is given by those maps having lifts
$$f_{a}(x,y)=(x+a\,y\,,\,y+\sin(2\pi\,(x+a\,y))),\ a\in\R,$$ 
and the variations in the following are given by those maps having lifts 
\begin{equation}
f_{\textbf{h},\textbf{v}}(x,y)=(x+\textbf{h}(y)\,,\,y+\textbf{v}\,(\sin(2\pi\,(x+\textbf{h}(y)))), \label{eq:variation-SF}
\end{equation}
where
\begin{itemize}
\item $\textbf{h}\in C^{0}(\R),\, [-1,1]\subset \textbf{h}([-1,1])$,
\item $\textbf{h}$ lifts a map of $\T^1$,
\item $\textbf{v}\in C^{0}(\R),\ \int_{[0,1]}\textbf{v}(\sin(2\pi\,x))\,dx=0$.
\end{itemize}
It turns out that these maps contain all the possible analytic diffeomorphisms on $\A$ which lift
maps of $\T^2$ having zero vertical drift. We write them in this way
so that they can be regarded as variations of elements of the SF.
Note that for such maps it holds $$N(f)=\max_{x\in[-1,1]} |\textbf{v}(x)|.$$
Although the CAP-part considers smooth maps $\textbf{h}$ and $\textbf{v}$, one could also consider here $\textbf{h}$ or $\textbf{v}$ being $C^0$ but having no further regularity than that. This would 
take us far from the traditional methods where differentiability is a key point. However, we do not demonstrate this here as the analytical case gives enough evidence of the efficiency of the method.

We have considered $f_{\textbf{h},\textbf{v}}$ with
\begin{equation}
	\textbf{h}(y)=y \qquad\textnormal{and}\qquad \textbf{h}(y)=\sin(2\pi\, y) \label{eq:h-choice}
\end{equation} 
for the twist case and for the non-twist case, respectively; together with
\begin{equation}
\begin{array}[c]{lll}
\textbf{v}(x)=x, &\quad  & \textbf{v}(x)=3\ln(x+2)+c,\\
\textbf{v}(x)=x(1-x)+\frac{1}{2}, &\quad  & \textbf{v}(x)=e^{x}-1+c,\\
\textbf{v}(x)=\tan(x) &  &
\end{array}\label{eq:v-choice}
\end{equation} 
for appropriate choices of $c$ so that $\int_{[0,1]}\textbf{v}(\sin(2\pi\,x))\,dx=0$ (see Tables \ref{tb:tableconservative1} and \ref{tb:tableconservative2}). We have obtained the following computer-assisted result, which is meant to show
the flexibility of the derived methods for proving chaos. 
\begin{thm}\label{thm:cap-vsf}
For all possible choices of {\em $\textbf{h}$} and {\em $\textbf{v}$} from (\ref{eq:h-choice}) and (\ref{eq:v-choice}) we have the existence of diffusion and rotational horseshoes for the map (\ref{eq:variation-SF}).
\end{thm}
A remarkable fact  here is that these CAPs can easily be automated and do not require an involved analysis of the prescribed map.

\subsection{The Non-Twist Standard Family}

Using Theorem \ref{thm:thm3} we obtain the following CAP of the existence of diffusion and rotational horseshoes for the NTSF (\ref{eq:NTSF-formula}) for a range of parameters. The choice of such region of parameters is done after visiting the literature, where usually the existence of a KAM tori inside the non-twist regime is investigated. The obtained parameters are meaningful in light of such previous works,
compare with Figure 9 of \cite{nontwist1}. 

\begin{thm} \label{thm:cap-ntsf}
    For at least $240\, 359$ parameter pairs $(a,b)$, chosen out of a uniformly distributed mesh of $1000\times 1000$
    points in the parameter domain $[0,1]\times [0,1]$, we have the existence of unbounded diffusion in the Non-Twist Standard Family (\ref{eq:NTSF-formula}). (See also Figure \ref{fig:ntsf}.)
\end{thm}

\begin{figure}
\begin{center}
	\includegraphics[width=6cm]{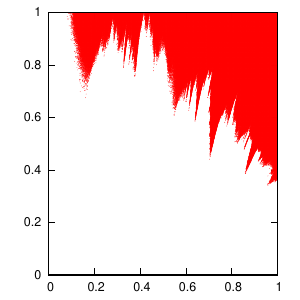}
\end{center}
\caption{Parameter pairs for which we have validated the existence of unbounded diffusion in the Non-Twist Standard Family (in red), from a mesh of $1000\times 1000$ points $(a,b)\in [0,1]^2$\label{fig:ntsf}. This can be copared with
Figure 9 of \cite{nontwist1}.}
\end{figure}

%\marginpar{GUYS:Here we should present the CAP?}

\subsection{Dissipative Standard Family}

Recall that the Dissipative Standard Family is given by those maps on $\A$ having lifts
\begin{equation}
    f_{a,b}(x,y)=(x+a\,y\,\,,\,\,b\,y+\sin(2\pi\,(x+a\,y))) \label{eq:DSF-formula}
\end{equation}
with $0<b<1$, being this value the determinant of the Jacobian of the map. 
In our setting, we assume $a>2$ to ensure that there are fixed points with distinct rotation vectors.
As before, we abuse notation calling the induced maps on $\A$ by the same name.
Applying Corollary B of \cite{paper1}  as described in Section \ref{s.dsfcap}, we obtained a computer-assisted proof of the existence of rotational horseshoes for the following range of parameters, which is taken under the following
consideration: the \emph{twist} parameter $a$ is taken large enough so to have fixed points with a rotational difference, avoiding to consider powers of the maps. The Jacobian $b$ is taken in a wide range, between 0.1 and 0.8, which shows
that the method works far from the strong dissipation regime, where the few existing rigorous proofs of the existence of chaos are usually considered.

%\marginpar{GUYS: Here we should present the CAP}
\begin{thm}\label{thm:cap-dsf}
In at least $95.95\%$ of the area of the parameter domain $(a,b)\in [3,10]\times[0.1,0.8]$ we have chaos for the Dissipative Standard Family (\ref{eq:DSF-formula}).
\end{thm}

The code for the computer-assisted proofs for Theorems \ref{thm:cap-vsf}, \ref{thm:cap-ntsf} and \ref{thm:cap-dsf} is available in \cite{code}.

%\subsection{Two Dimensional Arnold Family}

%Recall that the 2d Arnold Family is given by
%\begin{align*}
 %  f_{\tilde{f}_a,b}(x,y)\mapsto \left(\tilde{f}_a(x)+y\ ,\ b\,(\tilde{f}_a(x)-x+y)\right)
%\end{align*}
%where $0<b<1$, $\tilde{f}_a(x)=x+a\,\sin(2\pi\,x)+\omega$, $a\in\mathbb{R}$ and $\omega\in [0,1)$.

%\smallskip

%Once more, we abuse notation calling the induced maps in $\A$ by the same name.
%Based on Theorem A of \cite{}, applied as described in Section \ref{} we obtained the existence of rotational horseshoes
%for the following range of parameters.
%DSF

%\input{preliminaries} 
\section{Previous results}\label{s.prev}

We begin by stating the definition and a result from \cite{paper1}, which will serve as the theoretical starting point for this work. This includes introducing the key concept of a \emph{disjoint pair of neighborhoods}.
Recall that  $\textrm{Homeo}_0(\A)$ is the set of orientation and ends preserving
homeomorphisms. Consider two fixed points $x_0,x_1$ of $f\in\textrm{Homeo}_0(\A)$.
These points have different \emph{rotation vectors} if
$$(F(x_0)-x_0)-(F(x_1)-x_1)=(\rho,0)\mbox{ with }\rho\in \Z\setminus\{0\}$$
for any lift $F$ of $f$. Such number does not depend
on $F$ and hence is called \emph{rotational difference} of $x_0,x_1$.
 
\smallskip

Given $x_0$, $x_1$ as above, and $U_0,U_1$ close and connected neighborhoods of $x_0,x_1$ respectively,
we say that they form a $N$-\emph{disjoint pair of neighborhoods} for some $N\in\N$
whenever the sets
$$\bigcup_{j=0}^N f^j(U_0)\ \ \ \ \ \ \  ,\ \ \ \ \ \ \ \ \bigcup_{j=0}^N f^j(U_1)$$
are disjoint, and both inessential sets\footnote{Contained inside a topological disk}.
We denote such a pair of neighborhoods by $N$-d.p.n..

\small
\begin{figure}[htbp]
\centering
\def\svgwidth{.5\textwidth}
%% Creator: Inkscape 1.2.2 (732a01da63, 2022-12-09), www.inkscape.org
%% PDF/EPS/PS + LaTeX output extension by Johan Engelen, 2010
%% Accompanies image file 'Fig-Dpn.pdf' (pdf, eps, ps)
%%
%% To include the image in your LaTeX document, write
%%   \input{<filename>.pdf_tex}
%%  instead of
%%   \includegraphics{<filename>.pdf}
%% To scale the image, write
%%   \def\svgwidth{<desired width>}
%%   \input{<filename>.pdf_tex}
%%  instead of
%%   \includegraphics[width=<desired width>]{<filename>.pdf}
%%
%% Images with a different path to the parent latex file can
%% be accessed with the `import' package (which may need to be
%% installed) using
%%   \usepackage{import}
%% in the preamble, and then including the image with
%%   \import{<path to file>}{<filename>.pdf_tex}
%% Alternatively, one can specify
%%   \graphicspath{{<path to file>/}}
%% 
%% For more information, please see info/svg-inkscape on CTAN:
%%   http://tug.ctan.org/tex-archive/info/svg-inkscape
%%
\begingroup%
  \makeatletter%
  \providecommand\color[2][]{%
    \errmessage{(Inkscape) Color is used for the text in Inkscape, but the package 'color.sty' is not loaded}%
    \renewcommand\color[2][]{}%
  }%
  \providecommand\transparent[1]{%
    \errmessage{(Inkscape) Transparency is used (non-zero) for the text in Inkscape, but the package 'transparent.sty' is not loaded}%
    \renewcommand\transparent[1]{}%
  }%
  \providecommand\rotatebox[2]{#2}%
  \newcommand*\fsize{\dimexpr\f@size pt\relax}%
  \newcommand*\lineheight[1]{\fontsize{\fsize}{#1\fsize}\selectfont}%
  \ifx\svgwidth\undefined%
    \setlength{\unitlength}{589.47639807bp}%
    \ifx\svgscale\undefined%
      \relax%
    \else%
      \setlength{\unitlength}{\unitlength * \real{\svgscale}}%
    \fi%
  \else%
    \setlength{\unitlength}{\svgwidth}%
  \fi%
  \global\let\svgwidth\undefined%
  \global\let\svgscale\undefined%
  \makeatother%
  \begin{picture}(1,0.47619821)%
    \lineheight{1}%
    \setlength\tabcolsep{0pt}%
    \put(0,0){\includegraphics[width=\unitlength,page=1]{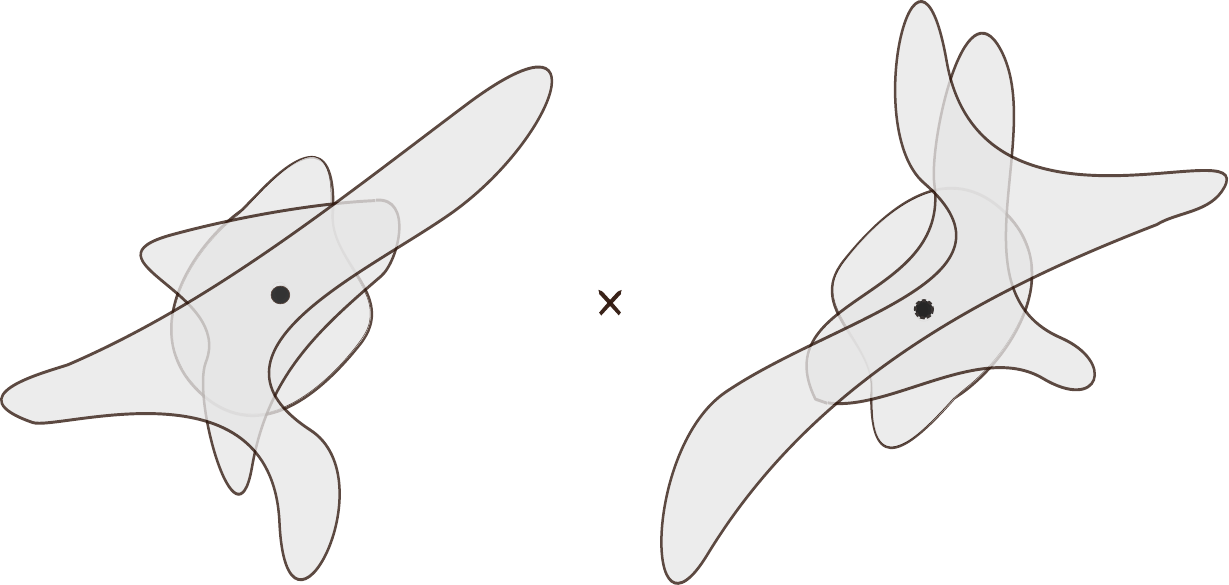}}%
    \put(0.29800717,0.14915056){\color[rgb]{0,0,0}\transparent{0.82469553}\makebox(0,0)[lt]{\lineheight{1.25}\smash{\begin{tabular}[t]{l}$U_0$\end{tabular}}}}%
    \put(0.05874992,0.31327225){\color[rgb]{0,0,0}\transparent{0.82469553}\makebox(0,0)[lt]{\lineheight{1.25}\smash{\begin{tabular}[t]{l}$f(U_0)$\end{tabular}}}}%
    \put(0.26858939,0.41292567){\color[rgb]{0,0,0}\transparent{0.82469553}\makebox(0,0)[lt]{\lineheight{1.25}\smash{\begin{tabular}[t]{l}$f^2(U_0)$\end{tabular}}}}%
    \put(0.70829939,0.06301861){\color[rgb]{0,0,0}\transparent{0.82469553}\makebox(0,0)[lt]{\lineheight{1.25}\smash{\begin{tabular}[t]{l}$U_1$\end{tabular}}}}%
    \put(0.88017352,0.11175264){\color[rgb]{0,0,0}\transparent{0.82469553}\makebox(0,0)[lt]{\lineheight{1.25}\smash{\begin{tabular}[t]{l}$f(U_1)$\end{tabular}}}}%
    \put(0.87419523,0.37105836){\color[rgb]{0,0,0}\transparent{0.82469553}\makebox(0,0)[lt]{\lineheight{1.25}\smash{\begin{tabular}[t]{l}$f^2(U_1)$\end{tabular}}}}%
  \end{picture}%
\endgroup%

\caption{A 2-d.p.n. for the pair of fixed points $x_0,x_1$. The annulus is represented as the plane minus a single point denoted by a cross.}
\end{figure}
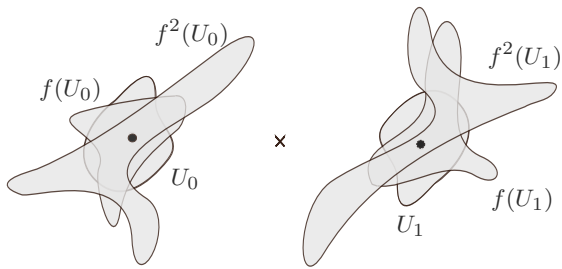
\normalsize

Finally, we say that two such fixed points $x_0,x_1$ are $N$-\emph{Birkhoff related}
whenever there exists a $N$-d.p.n. $U_0,U_1$ such that both the forward orbit
of $U_0$ intersects $U_1$ and the forward orbit of $U_1$ intersects $U_0$.

\smallskip

Let us briefly explain how the results of \cite{paper1} are applied in the present work.
In the Hamiltonian case, we do not directly use the main statements from \cite{paper1}. Instead, we obtain results analogous to Theorem B therein, relying on some of the techniques developed in that article (see Section~\ref{s.theoreticalr}) for which
the definitions above are crucial.
In contrast, in the dissipative case, we make direct use of Corollary B from \cite{paper1}, which we restate below.

\begin{corB}[Corollary B in \cite{paper1}]
Assume $f\in\textrm{Homeo}_0(\A)$ has a pair of
disjoint, essential and free curves $\gamma^-,\gamma^+$ so that the annulus $E$ bounded by these
verifies $f^n(E)\subset \textrm{int}(E)$ for some $n\in\N$ and contains two fixed points
$x_0,x_1$ of $f$ with rotational difference $\rho\in\N^*$. If for some $\left[\frac{34}{\rho}\right]$-d.p.n. $U_0,U_1$ of $x_0,x_1$, both the orbits of
$\gamma^-$ and $\gamma^+$
visit both $U_0$ and $U_1$, then the global attractor in $E$ contains a rotational horseshoe.
\end{corB}

\section{Proof of the theoretical results}\label{s.theoreticalr}

Let us start with some few definitions.
Recall that an \emph{annular conitnuum} $\mathcal{C}$ in $\A$ is a continuum so that 
 $\A\setminus \mathcal{C}$ is given by two essential and disjoint topological annuli accumulating in one of the ends. 
A circloid is given by an annular continuum which does not
properly contain another annular continuum. It is well known that every annular continuum contains some circloid.
If $\mathcal{C}$ is an essential annular continuum in $\A$, we define its \emph{vertical diameter} as 
$$VD(\mathcal{C})=\max_{x\in\mathcal{C}} \textrm{pr}_2(x) - \min_{y\in\mathcal{C}} \textrm{pr}_2(y).$$

For a homeomorphism $f$ of $\A$, we say that $f$ has unbounded upwards (resp. downward) diffusion if 
$$\sup_{n\in\N, x\in\A} \mathrm{pr}_2(f^n(x)-x)= \infty, \mathrm{(resp.} \inf_{n\in\N, x\in\A} \mathrm{pr}_2(f^n(x)-x)= -\infty.\mathrm{)}. $$
 $f$ is said to have unbounded diffusion if it has \emph{unbounded upward} (or \emph{downward}) \emph{diffusion}.
 In any of the two cases it holds
$$\sup_{n\in\Z, x\in\A} \abs{\mathrm{pr}_2(f^n(x)-x)}= \infty,$$
We say that $f$ has \emph{bounded diffusion} if  it does not have unbounded diffusion.
In this case we can define the value $M(f)\in\R$ as the supremmun considered in the last equation.
Note that if $f$ is a non-wandering map of $\A$, then it has upward diffusion if and only if it has downward diffusion.

\smallskip

If $f\in \textrm{Homeo}_{\,\textrm{Dehn}_k,\textrm{nw}}(\A)$  or if $f\in \textrm{Homeo}_{\,0,\textrm{nw}}(\A)$, then one can easily show  that if $f$ has bounded diffusion, it must also be non-wandering in $\A$. Indeed, $f$ lifts a homeomorphism $\check f$ of $\T^2$ that is nonwandering. If $U\subset \A$ is any ball with radius less than $1/4$, then $U$ projects injectively in a ball $\check U \subset \T^2$ and for every $M$ there exists some point $\check x\in \check U$ and integers $0=i_1<i_2<\hdots<i_n$ such that $\check f^{i_j}(\check x)\in \check U$. If $x$ is a lift of $\check x$ in $U$, then one finds integers $w_j, 1\le j\le n$ such that $f^{i_j}(x)\in U+(0,w_j)$, and if $f$ has bounded diffusion with contant $M(f)<n/2$ then $w_{j_1}=w_{j_2}$ for some $j_1< j_2$, showing that $U+(0,w_{j_1})$ is nor wandering and consequently neither is $U$. 

\smallskip

We have the following proposition.

\begin{prop}\label{pr:horseshoezerodrift}
Let $f\in \textrm{Homeo}_{\,\textrm{Dehn}_k,\textrm{nw}}(\A)$  or $f\in \textrm{Homeo}_{\,0,\textrm{nw},\rho}(\A)$ with $\rho\ge 1$ preserve the Lebesgue measure $\lambda$, and assume that $f$ has unbounded diffusion and zero mean vertical drift. Then $f$ has a rotational horseshoe.
\end{prop}
\begin{proof}
First let us assume that $f\in \textrm{Homeo}_{\,\textrm{Dehn}_k,\textrm{nw}}(\A)$. The hypothesis on the average vertical displacement of $f$ implies that the vertical rotation number of the Lebesgue measure is null (see \cite{addas2012dynamics} for the definition) and as shown in Corollary 2 of that paper, $0$ must belong to the interior of the \emph{horizontal rotation} set of $f$ (see definition int the mentioned paper, when is used the \emph{vertical} case), which implies that $f$ must have both unbounded diffusion upward and and unbounded diffusion downward. Furthermore, by Theorem~2 of \cite{addas2005some}, there must exists some $x_0$ in $\A$ that is fixed. Taking any lift $\tilde f$ of $f$ to $\R^2$ such that $x_0$  lifts to a fixed point $\tilde x_0$ for $\tilde f$, the twist hypothesis shown that $\tilde f(\tilde x_0+(0,1))= \tilde x_0+(k,1)$, so that $f$ has fixed points of different rotation number. Since $f$ lifts a homeo of $\T^2$ and has unbounded diffusion both upwards and downwards, there are points arbitrary low on $\A$ with positive iterates that are arbitrarily high, and likewise there are points arbitrary high on $\A$ with positive iterates that are arbitrarily low. The result then follows directly from Proposition D of \cite{lecalveztal}.

The proof for the case where $f\in \textrm{Homeo}_{\,0,\textrm{nw},\rho}(\A)$ with $\rho\ge 1$ is similar. We can assume that $f$ lifts a torus map $\check f :\T^2\to\T^2$, and in turn $f$ has a lift to $\tilde f$ to $\R^2$ where $\tilde f$ has a point $\tilde x_0$ satisfying $\tilde f(\tilde x_0)=\tilde x_0$ and a point $\tilde x_1$ satisfying $\tilde f(\tilde x_1)=\tilde x_1+(0,\rho)$.  If $\rho_{MZ}(\tilde f)$ is the Misiurewicz-Zieman rotation set for $\tilde f$ as a lif of $\check f$, then the fact that $(0,0)$ and $(\rho,0)$ belong to this rotation set and that $f$ has unbounded diffusion implies by \cite{nakotal, davalos} that $\rho_{MZ}(\tilde f)$ has non-empty interior, and the solution of Boyland's conjecture in \cite{forcing} shows that the rotation vector of Lebesgue, which is of the form $(a,0)$ belongs to the interior, from where we deduce again that $f$ has unbounded diffusion both upward and downward. The rest of the proof follows as in the previous case, as we already know the existence of 2 fixed points for $f$ with different rotation number.
\end{proof}

\subsection{Diffusion for periodic conservative maps homotopic to Dehn twists}

Let $\textrm{Homeo}_{\,\textrm{Dehn}_k}(\A)$ be the set of homeomorphisms of $\A$ that preserve orientation and have a lift to $\R^2$ satisfying $\tilde f (\tilde x+(0,1))= \tilde f (\tilde x)+(k, 1)$, where $k$ is a positive integer. These maps are lifts of homeomorphisms of $\T^2$ isotopic to the linear Dehn twist induced by $D_k(x,y)= (x+ky, y)$. Let also $\textrm{Homeo}_{\,\textrm{Dehn}_k,\textrm{nw}}(\A)$ be the subset characterized by those lifts of non-wandering maps in $T^2$.
We want to prove Theorem  1, which is given by the following statement. 

\begin{thm}\label{thm:difussionconservative twist}
Let $f\in \textrm{Homeo}_{\,\textrm{Dehn}_k,\textrm{nw}}(\A)$ and let $M=\min\{3N(f)+2, N(f)+4\}$. If there exists some $x\in \A$ and $n\in\Z$ such that $\abs{\mathrm{pr}_2(f^n(x)-x)}\ge M$, then $f$ has unbounded diffusion. Moreover,
if $f$ preserves an absolutely continuous measure having has zero vertical drift, then $f$ has rotational chaos.
\end{thm}

Before proceeding with the proof, let us start with a lema that is very similar to Lemma~8.3 of \cite{paper1}.

\begin{lema}\label{lemaessentialsubset}
Let $A\subset \A$ be an $f$-invariant, essential, open and bounded topolological annulus. Assume that the prime end rotation number of the lower boundary is $\rho$, and the prime ends rotation number of the upper boundary is $\rho+p$ for some integer $p$. Let $\sigma:[0,1]\to \A$ be a curve that joins the upper and lower boundary of $\A$, and is contained in $A$ but for its endpoints. Then $\sigma\cup f^3(\sigma)$ contains an essential compact subset that is contained in $A$. Furthermore, if $p=2$ then we can show the same for $\sigma\cup f^2(\sigma)$, and if $p\ge 3$ then we can show the same for $\sigma\cup f^2(\sigma)$.
\end{lema}
\begin{proof}
The proof is essentially the same as that  of \cite{paper1}.  One can go to the prime ends compactification of $A$, and obtain a homeomorphism $h:A\to \T\times (0,1)$ such that $g=h\circ f \circ h^{-1}$ extends to a continuous homeomorphism of $\T\times [0,1]$, which we still denote $g$, and that has a lift $\tilde g:\R\times [0,1] \to \R\times [0,1]$ such that the rotation number for $\tilde g$ of $\T^1\times\{0\}$ is $\rho$ and of the upper boundary is $\rho+p$. We further assume, with no loss in generality, that $0\le 3 \rho<1$.

Now $h\circ \sigma\mid_{(0,1)}$ extends continuously to an arc $\beta:[0,1]\to \T\times [0,1]$ joining the two boundaries, and let $\tilde \beta$ be a lift of this arc. Denote $z_0=\tilde \beta(0) = (x_0, 0)$ and $z_1=\tilde \beta(1) = (x_1, 1)$. The complement of $\tilde \beta$ has two unbounded components, one to the left that contains $(-\infty,x_0)\times\{0\}$ and $(-\infty,x_1)\times\{1\}$, and one to the right that contains $(x_0, \infty)\times\{0\}$ and $(x_1, \infty)\times\{1\}$. One verifies that, by the assumption on the rotation number, that $\tilde g^3 (z_0)=(y_0,0)$ with $y_0<x_0+1$ and $\tilde g^3 (z_1)=(y_1,1)$ with $y_1>x_1+2$. This implies that $\tilde g^{3}(\tilde \beta)$ intersects both $\tilde \beta+(1,0)$ and $\tilde \beta +(2,0)$, so that $\beta\cup g^3(\beta)$ has an essential compact set in the interior of $\T\times[0,1]$ and the result follows by taking $h^{-1}$ of this set. The proofs for the cases where $p=1$ or $p+2$ are similar, noting that in these cases just noting that, in the first case, the rotation number of $g$ for the upper boundary is at least $\rho+3$, and in the second case, the rotation number for $g^2$ for the upper boundary is at least $\rho+4$.  
\end{proof}

\begin{lema}
If $f\in \textrm{Homeo}_{\,\textrm{Dehn}_k,\textrm{nw}}(\A)$ has bounded diffusion, then $f$ has an invariant circloid.  
\end{lema}
\begin{proof}
If $M= \sup_{n\in\Z, x\in\A} \abs{\mathrm{pr}_2(f^n(x)-x)}<\infty$, then the set 
$$A'=\bigcup_{i\in \Z} f^i(\pi(\R\times\{0\}))$$ 
is an open, essential, $f$ invariant set contained in $\pi(\R\times[-M, M+1])$ and so it $A=\textrm{Fill}(A')$, which is the union of $A'$ with all bounded connected components of its complement. Therefore $A'$ is an essential open annulus, and its boundary contains exactly 2 circloids, which are invariant.
\end{proof}

A third lemma follows.

\begin{lema}\label{lemadiamcircloid}
If $f$ has bounded diffusion and $\mathcal{C}$ is an essential $f$ invariant circloid, then $\sup_{n\in\Z, x\in\A} \abs{\mathrm{pr}_2(f^n(x)-x)}\le VD(\mathcal{C})+1$.
\end{lema} 
\begin{proof}

Define 
$$D=\mathcal{U}^{-}(\mathcal{C})\cup \mathcal{U}^{+}(\mathcal{C}+(0,1))$$ 
For every $x\in \A$ there exists an integer $l$ such that $x+(0,l)$ lies in $D$ which 
is invariant and has vertical diameter $VD(\mathcal{C})+1$. As the orbit of $x+(0,l)$
can not scape $D$ we obtain the conlcusion of the Lemma.
\end{proof}

\begin{proof}[Proof of Theorem \ref{thm:difussionconservative twist}.]

We will show that if $f$ has bounded diffusion, then 
$$\abs{\mathrm{pr}_2(f^n(x)-x)}< M,$$
with $M$ as defined in the Theorem. Indeed, if $f$ has bounded diffusion, it has an invariant circloid $\mathcal{C}$. Let $\tilde f$ be a lift of $f$. Since $f$ lifts a non-wandering map of $\T^2$, every point in $\mathcal{C}$ has the same rotation number (Theorem A of \cite{koropeki}) for $f$, which we assume is $\rho$. Since the rotation number for $x+(0,p),\,x\in \mathcal{C}$ is $\rho+k\,p$ for every positive integer $p$ (by the \emph{Dehn hypothesis}), then $\mathcal{C}$ and $\mathcal{C}+(0,p)$ are always disjoint. 

First consider the case $p=1$. There exists an invariant open annulus $A$ contained above $\mathcal{C}$ and below $\mathcal{C}+1$. We assume, with no loss of generality, that $\min_{y\in\mathcal{C}} \textrm{pr}_2(y)=0$, otherwise we just do a change of coordinates. If $VD(\mathcal{C})<1$, then $M<2<2+N(f)$ because of Lemma~\ref{lemadiamcircloid}. If not, then there exists a horizontal line segment $\sigma:[0,1]\to \pi(\R\times\{1\})$ joining  $\mathcal{C}$ and $\mathcal{C}+1$ and contained in $A$ but for its endpoints. As deduced above, the prime ends rotation number of the lower boundary of $A$ is $\rho$, and the prime ends rotation number of the upper boundary of $A$ is  grater or equal than $\rho+1$. Thus Lemma~\ref{lemaessentialsubset} show that $\sigma\cup f^3(\sigma)$ contains a compact essential subset of $A$. But this compact set separates $\mathcal{C}$ from the upper end of $\A$, and since there exists some $x=(VD(\mathcal{C}), \theta)$ in $\mathcal{C}$, there must be some $y\in \sigma$ such that if $z=f^3(y)$, then $z=(a, \theta)$, with $a>VD(\mathcal{C})$, and so $\mathrm{pr}_2(f^3(y)-y)> VD(\mathcal{C})-1$.  Since $\mathrm{pr}_2(f^3(y)-y)\le 3 N(f)$, we deduce that $VD(\mathcal{C})<3N(f)+1$ and so $M<3N(f)+2$. The other bound on $M$ is obtained in the same way, but using the annulus bounded by $\mathcal{C}$ and $\mathcal{C}+(0,3)$.

\smallskip

This way we obtain the sufficient condition for having unbounded diffusion. For the existence of the rotational
horseshoe in the area preserving situation one employees Proposition  \ref{pr:horseshoezerodrift}.
\end{proof}

\subsection{Rotation sets with non-empty interior for area preserving homeomorphisms}

In this sections we deal with homeomorphisms $f$ of the annulus that lift non-wandering homeomorphisms of the $2$-torus in the isotopy class of the identity, that is, we assume that $f\in \textrm{Homeo}_{\,0,nw}(\A)$.  Recall from the introduction that if $f$ is in $\textrm{Homeo}_{\,0,nw,\rho}(\A)$, then, there exists a lift $\tilde f:\R^2\to\R^2$ of $f$, a positive integer $\rho$ and two fixed points $x_0, x_1$ in $\A$ such that, if $\tilde x_0$ and $\tilde x_1$ are lifts of $x_0$ and $x_1$ respectively, then $\tilde f(\tilde x_0)=\tilde x_0$ and $\tilde f(\tilde x_1)=\tilde x_1+(0,\rho)$.

\begin{thm}\label{th:rotationsetwithinterior}
Let $f\in \textrm{Homeo}_{\,0,nw,\rho}(\A)$ and 
$$M_1= 6N(f)+2\ ,\ M_2=4N(f)+2\ ,\ M_i =2N(f)+2\mbox{ for }i\ge 3.$$ 
If for  some $x\in \A$ it holds 
$$\abs{\mathrm{pr}_2(f^n(x)-x)}\ge M_{\rho},$$
then $f$ has unbounded diffusion. In particular, $f$ projects to an homeomorphism $\check f$ of $\T^2$
whose rotation set has non-empty interior.

Furthermore, in case $f$ preserves an absolutely continuous measure $\lambda$ and
 has zero vertical drift, then $f$ has a topological horseshoe and total diffusion. 
\end{thm}

Before proving the theorem, we will show the following lemmas. Let us introduce some standar notation.
An annular continua $\mathcal{C}$ has the two unbounded connected components of the complements.
Call $\mathcal{U}^+(\mathcal{C})$ to that one for which the restriction of $\textrm{pr}_2$ is bounded below,
and call $\mathcal{U}^-(\mathcal{C})$ to the other one. Given two annular continua 
$\mathcal{C}_1,\mathcal{C}_2$ we say that $\mathcal{C}_2$ is above $\mathcal{C}_1$ iff 
$\mathcal{C}_2$ is contained in $\mathcal{U}^+(\mathcal{C}_1)$.

\begin{lema}
If $f\in\textrm{Homeo}_{\,0,nw,\rho}(\A)$ and has bounded diffusion, then there exists an circloid $\mathcal{C}$ that is invariant by $f$. Furthermore, $\mathcal{C}$ is disjoint from $\mathcal{C}+(0,1)$. 
\end{lema}
\begin{proof}
This is mostly already done in Theorem 5.1 of \cite{nakotal}. Since $f$ has bounded diffusion, there exists $M>0$ such that, for all $x\in\A$ and all $n\in\Z$, $\abs{\mathrm{pr}_2(f^n(x)-x)}< M$. Since $f$ lifts an homeomorphism $\check f$ of $\T^2$ homotopic to the identity, and $\rho>0$, then $\tilde f:\R^2\to\R^2$ also lifts $\check f$ and its Misiurewicz-Zieman rotation set as a lift of $\check f$ must be a nondegenerate horizontal segment containing $[0,\rho]\times\{0\}$. Theorem 5.1 of \cite{nakotal} then applies, and there exists some $\check f$ invariant horizontal annulus $\check A$ in $\T^2$. If $A$ is a lift of $\check A$ to $\A$, then $A$ is an $f$ invariant essential annulus, disjoint from its integer translates. There exists exactly two circloids on the boundary of $A$, $\mathcal{C}_{-}$ and $\mathcal{C}_{+}$ which are disjoint, with $\mathcal{C}_{-}$ below $\mathcal{C}_{+}$. Furthermore, $\mathcal{C}_{+}-(0,1)$ is not above $\mathcal{C}_{-}$, so $\mathcal{C}_{+}-(0,1)$ and $\mathcal{C}_{+}$ must be disjoint, which implies that, if $\mathcal{C}=\mathcal{C}_{+}$, then  $\mathcal{C}$ is disjoint from all its translates.
\end{proof}

%\begin{proof} The proof here follws very closely the one done at \cite{nakotal}, THIS IS A SKETCH: WILL look on GKT cited below, think this is done already  $\mathcal{C}$ has a single rotation number for $\hat f$, and its the same as the rotation number of $T^{i}(\mathcal{C}$ for all integers $i$, since we assumed $f$ lifts $\check f$. Now, either $\mathcal{C}$ is disjoint from $T(\mathcal{C})$, or one can show that the complement of $\bigcup_{i\in\Z} T^i(\mathcal{C})$ is a union of bounded topological disks. The map $f$ being nonwandering, its a straightforward proof (See Guelman, Koropecki, Tal) that all recurrent points in the complement of this union have rational otation number. But its is known that for every irrational in $(0,\rho)$ there must be a recurrent point with that irrational rotation number, a contradiction  \end{proof}

We are ready to proceed with the main result of the subsection.

\begin{proof}[Proof of Theorem~\ref{th:rotationsetwithinterior}]
 
We already know from the previous lemma that $f$ has an invariant circloid $\mathcal{C}$ which is disjoint from its integer vertical translates. In particular, the orbit of every point is bounded (as it is contained in an invariant annulus whose boundary is made of 2 different integer translates of $\mathcal{C}$), and since $f$ lifts a nonwandering homeomorphism of $\T^2$ this implies that $f$ itself must be nonwandering. As a consequence, since every circloid for a nonwandering homeomorphism of $\A$ has the same rotation number for $\tilde f$ (see, for instance, Theorem~C of \cite{paper1}), every point of $\mathcal{C}$ has the same  rotation number. We will treat the case where its rotation number is larger or equal to $\rho/2$, where in particular, $x_0$ does not belong to any integeger translate of $\mathcal{C}$.  The case where its smaller than 
$\rho/2$ can be treated analogously, but using $x_1$ instead of $x_0$.

We assume, with no loss in generality, that $x_0$ lies in the topological annulus $A$ defined by $\mathcal{U}^{+}(\mathcal{C}) \cap \mathcal{U}^{-}\left(\mathcal{C}+(0,1)\right)$. Also, we may assume that $\min_{y\in\mathcal{C}} \textrm{pr}_2(y)= 0$. It follows that $0<\textrm{pr}_2(x_0)<VD(\mathcal{C})+1$. One can find an horizontal arc $\sigma:[0,1]\to\pi(\R\times\{\textrm{pr}_2(x_0)\})$, such that $\sigma\mid_{[0,1)}$ lies in $A$ with $\sigma(0)=x_0$ and with $\sigma(1)$ in either $\mathcal{C}$ or in $\mathcal{C}+(0,1)$. We claim that $\sigma\cup f^n(\sigma)$ contains a compact essential subset if $2/\rho <n$. Indeed, if this was not the case, then there would exist a neighborhood $U$ of $\sigma$ such that $\bigcup_{i=0}^{n} f^{i}(U)$ is inessential and contains $x_0$. By Proposition 4.9 of  \cite{paper1} this implies that one of the prime ends rotation number for $\mathcal{C}$ lies in the interval $[-1/n, 1/n]$ a contradiction since this interval does not intersect $[\rho/2, \infty)$. 

But this implies that $\sigma\cup f^n(\sigma)$ must intersect both the lines 
$\pi(\R\times\{1\})$ and $\pi(\R\times \{VD(\mathcal{C})\})$. One obtains that either 
$$\textrm{pr}_2(x_0)\le(VD(\mathcal{C})+1)/2\ \mathrm{  and }\  \max_{z\in f^n(\sigma)}\textrm{pr}_2(z)>VD(\mathcal{C}),$$ which implies that $n\cdot N(f)>(VD(\mathcal{C})-1)/2$ or similarly that  
$$\textrm{pr}_2(x_0)\ge(VD(\mathcal{C})+1)/2\ \mathrm{  and }\ \min_{z\in f^n(\sigma)}\textrm{pr}_2(z)<1,$$ 
again obtaining that $n\cdot N(f)>(VD(\mathcal{C})-1)/2$.
Hence, we deduce that $VD(\mathcal{C})<2n\,N(f)+1$. Now the same argument as in Lemma~\ref{lemadiamcircloid} shows that 
$$\sup_{n\in\Z, x\in\A} \abs{\mathrm{pr}_2(f^n(x)-x)}\le VD(\mathcal{C})+1 <2n\,N(F)+2$$

 Note that if $\rho=1$ we must take $n\ge3$ to satisfy $2/\rho <n$, which yelds $M_1=6N(f)+2$, while if $\rho=2$ we may take $n=2$ and if $\rho\ge 3$ we can take $n=1$, yelding that $M_2=4N(f)+2$ and $M_\rho=2N(f)+2$ if $\rho\ge 3$, proving the first assertion of the result.

Finally, the existence of horseshoes for the zero mean vertical drift then follows directly from Proposition~\ref{pr:horseshoezerodrift}
\end{proof}

\subsection{Diffusion for the Non-Twist Standard Family}

We consider the family of homeomorphisms in $\A$ 
having as lifts
% $$ f_{a,b}(x,y)= (x+(y+a\, sin(2\pi x))^2+b \, ,\, y+a\,sin(2\pi x)),$$ 
$$f_{a,b}(x,y)=\left(x+a\,\left[ 1- ( y-b\,\sin(2\pi\,x))^2\right],  y-b\,\sin(2\pi\,x)\right)$$

which is called in the literature the \emph{Non-Twist Standard Family}. Let us  abuse notation, by calling also
$f_{a,b}$ the maps induced on $\A$. Note that these maps preserve the Lebesgue measure on $\A$ and fail in having the twist condition. This last issue, prevents this family of being treated in a global sense with
those tools largely extended in the twist realm. The idea is to show how the results and techniques
introduced here permits to treat this case.

\smallskip

The aim of this section is to determine for this particular family theoretical bounds on the maximal vertical displacement that a map can have while still having bounded diffusion. Namely, we will provide a bound $M_{a,b}$ so that, if there exists some $z\in \A$ and some integer $n$ for which 
$$\abs{\mathrm{pr}_2(f_{a,b}^n(z)-z)}\ge M_{a,b}$$ then $f_{a,b}$ has unbounded diffusion.
Note that it holds  
$$N(f_{a,b})=\max_{z\in\A}\{\abs{\mathrm{pr}_2(f_{a,b}(z)-z)}\}=|b|$$
We want to  prove the following result. 

\begin{thm}\label{th:boundsfornontwiststandardfamily}
Let $a \in (0,\infty)$ and $b\in\R$ be given, and set 
%$$M_a= \max\left\{ 3\,a +\frac{2}{a}\, ,\, 2\,a +\frac{4}{a}\right\}$$
\begin{equation}
	M_{a,b}=\,\max\left\{  2\,\sqrt{1+\frac{3}{a}}+|b|\,,\,\frac{2}{a\,|b|}\right\}. \label{eq:Mab-ntsf}
\end{equation}
Then,
\begin{enumerate}

\item For all parameters $a \in (0,\infty)$ and $b\in\R\setminus\{0\}$ the map $f_{a,b}$ has a rotational horseshoe.

\item If there exists a point $z\in\A$ with $\mathrm{pr}_2(z)<-M_{a,b}$ and an integer $n$ such that $\mathrm{pr}_2(f^n_{a,b}(z))>M_{a,b}$, then $f_{a,b}$ has total unbounded diffusion.
\end{enumerate}
 \end{thm}

Note that because of the symmetry of the expressions with respect to take opposite vectors,
proving the existence of unbounded diffusion implies total unbounded diffusion.
We will provide the proof of the theorem in the remaining part of the section
and work first on the proof of point (2) as the required tools are similar 
to those employed above. Let us assume that $a,b$ are already fixed and start with the following lemmas. 

The first lemma is identical to previous results on the paper and we omit the proof.
\begin{lema}\label{l.circ}
If $f_{a,b}$ has bounded diffusion, then $f_{a,b}$ has an invariant circloid $\mathcal{C}$.
\end{lema}

The following lemmas refer to properties of invariant circloids for the maps $f_{a,b}$.
For an invariant circloid $\mathcal{C}$ of $f_{a,b}$ consider the values
$$L_1(\mathcal{C})=\min_{z\in \mathcal{C}}\mathrm{pr}_2(z)\ \ \ ,\ \ \  L_2(\mathcal{C})=\min_{z\in\mathcal{C}}\mathrm{pr}_2(z).$$
In the following, we denote for $i=1,2$ the set $L_i(\mathcal{C})$ by $L_i$ in order to reduce notation.
We have the following lemma.

\begin{lema}\label{l.menorqueb}
$L_2-L_1\ge |b|$.
\end{lema}
\begin{proof}
There exists $y$ such that the projection of $(1/4,y)$ belongs to $\mathcal{C}$.
 But $f_{a,b}(y)=(x', y+|b|)$ also belongs to $\mathcal{C}$ and so $L_1\le y\le y+|b|\le L_2$.
\end{proof}

The next lemma gives a first restriction about how large $L_2-L_1$ can be.

\begin{lema}\label{l.boundnontwist}
The map $f_{a,b}$ cannot have a pair of fixed points $z_0$ and $z_1$ with rotational difference of 3 and contained in 
$$\textrm{pr}_2^{-1}(\,[L_1+|b|, L_2-|b|]\,).$$
\end{lema}

The proof is very similar to what has been done above.

\begin{proof}
Assume for a contradiction this is false. Since $f_{a,b}$ is conservative and $\mathcal{C}$
is a circloid, as consequence of Theorem C of \cite{paper1} shows
that the rotation set on $\mathcal{C}$ is reduces to a singleton for $f_{a,b}$ (recall that we name the lift and the
induced maps by the same name).
In the assumed situation one can join $z_0$ to $\mathcal{C}$ by a possibly degenerate horizontal line segment  
$\sigma_0$ intersecting $\mathcal{C}$ only at one endpoint which we denote $p_0$. If this line segment is degenerate, then $z_0=p_0$ belongs to $\mathcal{C}$ and $\rho=\rho(f_{a,b}, z_0)$ is the rotation number of $\mathcal{C}$. If not,
then $\sigma_0\setminus\{p_0\}$ is disjoint from $\mathcal{C}$. Since $N(f)=|b|$ and $p_0$ is a fix point, then image of $\sigma_0$ is contained in $\textrm{pr}_2^{-1}(\,[L_1, L_2]\,)$. Since it is disjoint from $\mathcal{C}$, this means that $\sigma_0\cup f_{a,b}(\sigma_0)$ is inessential. In this situation we can apply  Proposition 4.9 of \cite{paper1} and obtain that one of the priem-end rotation numbers of $\mathcal{C}$ lies in $[\rho( f_{a,b}, z_0)-1, \rho( f_{a,b}, z_0)+1]$ and so $\abs{\rho-\rho( f_{a,b}, z_0)}\le 1$.
An analogous argument with $z_1$ instead $z_0$ shows that  $\abs{\rho-\rho(f_{a,b}, z_1)}\le 1$.
This way we arrive to a contradiction, since $\rho(f_{a,b}, z_1)-\rho(f_{a,b}, z_0)\ge 3$, and this implies that $\mathcal{C}$ has a non-trivial interval as rotation set (\cite{luishernandez}).
\end{proof}

Our next lemma takes care of those invariant circloids contained in either $\textrm{pr}_2^{-1}(x,y)>0$ or 
$\textrm{pr}_2^{-1}(x,y)<0$. In such situation, we have a twist map and then, due to Birkhoff classical results \cite{birk3},
 the circloid must be given by the projection of a graph in $\R^2$.

\begin{lema}\label{l.boundgraph}
If $L_1>0$, then $L_2-L_1< \frac{2}{a\,(L_1+L_2)}$ and if $L_2<0$ then $L_2-L_1< \frac{2}{a\,(|L_2|+|L_1|)}$.
\end{lema}
\begin{proof}
We first note that, if $L_1>0$, the restriction of $f_{a,b}$ to $\textrm{pr}_2^{-1}(\,(L_1,+\,\infty)\,)$ is a $\mathcal{C}^{\infty}$
conservative twist map, and one can extend it to a map $g_{a,b}:\A\to\A$ that is still a $\mathcal{C}^{\infty}$ conservative 
twist map. Of course $\mathcal{C}$ is also an invariant circloid for $g_{a,b}$. Then, we can apply 
Birkhoff curve Theorem, see \cite{birk3}, which states than in this case $\mathcal{C}$ must be the projection of graph 
a graph of a Lipschitz function $\tilde{\theta}:\R\to\R$ which we still call $\tilde{\theta}$.

Note that map $g_{a,b}$ is a composition of $\tau_{a}\circ \eta_{b}$ with
$$\tau_{a}(x,y)=\left(x+a\,[1-y^2]\,,\,y\right)\ \ \ ,\ \ \ \eta_{b}(x,y)=\left(x\,,\,y-b\,\sin(2\pi\,x)\right)$$
It holds that $\eta_b$ transform any graph in a new graph. Moreover, as $\tau_a$ preserves the
horizontal lines in $\R^2$, it must be the case that the gaphs $\tilde{\theta}$ and $\eta_b(\tilde{\theta})$ 
has exactly the same image (as functions) given by $[L_1,L_2]$. On the other hand, one notes that
for $\tau_a$ transform a graph in $\R\times\R^+$ into a graph, it is required that the collection of rotation numbers
induced by $\tau$ on the (projection of the) horizontal lines met by the graph, have a difference bounded above by $2$:
Otherwise, considering the curve $\tilde{\alpha}$ inside the graph $\tilde{\theta}$ in-between points $(x_1,L_1)$
and $(x_2,L_2)$ with $x_1,x_2\in [0,1)$, one obtains that $g_{a,b}(\tilde{\alpha})$ is given by a curve
which contains a fundamental domain of $\theta$, being this impossible since $\tilde{\alpha}$ does not contain any
fundamental domain. Thus we
have the condition
$$a\,[1-L_1^2]-a\,[1-L_2^2]<2\mbox{, that is }$$
 $$a\,L_2^2-a\,L_1^2 <2\mbox{ so }$$
$$L_2-L_1<\frac{2}{a\,(L_1+L_2)}$$
 The case $L_2<0$ is obtained by the just covered case, due to the symmetry of the map.
\end{proof}

Merging this last lemma with Lemma \ref{l.menorqueb}, we have the following corollary.

\begin{cor}\label{c.grahps}
There is no invariant circloid for $f_{a,b}$ with $L_1>\frac{2}{a\,|b|}$ or $L_2<-\frac{2}{a\,|b|}$.
\end{cor}

In contrast with the last two lemmas, the next one takes care of the non-twist region, for which we use
Lemma \ref{l.boundnontwist}.

\begin{lemma}\label{l.nontwistregion}
If $L_1<0$ and $L_2>0$, then $L_2-L_1<2\,\sqrt{1+\frac{3}{a}}+2\,|b|$. 

\end{lemma}

\begin{proof}

Computing those points $p_{\kappa}=(x_{\kappa},y_{\kappa})$
whose image under $f_{a,b}$ goes to $(x_{\kappa}+\kappa,y_{\kappa})$ for $\kappa\in \Z$ one finds
the sub-family 
$$\pm p_{\kappa}=\left(0 ,  \pm \sqrt{1-\frac{\kappa}{a}}\right),\ \kappa\in\Z\setminus\N$$
Assume for a contradiction that $L_2-L_1\geq 2\,\sqrt{1+\frac{3}{a}} +2\,|b|$. Hence  
$L_2-L_1-2\,|b|\geq 2\,\sqrt{1+\frac{3}{a}}$. Then,
$$\{ p_0 \,,\,  p_{-1} \,,\,  p_{-2} \,,\, p_{-3} \}\subset  \textrm{pr}_2^{-1}(\,[L_1+|b|, L_2-|b|]\,)\mbox{ or }$$
$$\{ -p_0 \,,\,  -p_{-1} \,,\,  -p_{-2} \,,\, -p_{-3} \}\subset  \textrm{pr}_2^{-1}(\,[L_1+|b|, L_2-|b|]\,)$$
In both cases one finds a contradiction to Lemma \ref{l.boundnontwist}.
\end{proof}

We are ready to prove Theorem 3. As we mentioned, we start by point (2) as the involved techniques which
were developed above fits with the previous kind of results.  

\begin{proof}[Proof of point (2) of Theorem 3]

Assume for a contradiction that there is an invariant circloid for $f$ and consider the values $L_1<L_2$. We discuss the following cases.

\begin{enumerate}

\item $L_1>0$: Corollary \ref{c.grahps} implies that $L_1\leq \frac{2}{a\,|b|}$. Then, considering  
Lemma \ref{l.menorqueb} and with Lemma \ref{l.boundgraph}, one obtains that
$$L_2\leq \frac{2}{a\,|b|} + \frac{2}{a\,|b|}=\frac{4}{a\,|b|}$$
Hence the circloid must be contained in
$$\textrm{pr}_2^{-1}\left(\left[0\,,\,\frac{4}{a\,|b|} \right] \right)$$

\item $L_2<0$: on obtains by symmetry that $L_1\geq -\frac{4}{a\,|b|}$, so
$$\textrm{pr}_2^{-1}\left(\left[-\frac{4}{a\,|b|} \,,\,0\right]\right)$$

\item $L_1<0<L_2$: Lemma \ref{l.nontwistregion} implies that the circloid is contained in 
$$\textrm{pr}_2^{-1}\left(\left[-\left(2\,\sqrt{1+\frac{3}{a}}+2\,|b|\right)\,,\,2\,\sqrt{1+\frac{3}{a}}+2\,|b|\right]\right)$$

\end{enumerate}

Hence, if some point in $\textrm{pr}_2^{-1}(-\infty,-M_{a,b})$ visits  $\textrm{pr}_2^{-1}(M_{a,b},+\infty)$
there is no invariant circloid, and the map presents diffusion. Moreover, due to the symmetry of the map,
$f_{a,b}$ will present total diffusion, and then applying the main result in \cite{lecalveztal}, being the
two ends of $\A$ Birkhoff related, we obtain the existence of a rotational horseshoe. 
\end{proof}

We finish the section with the proof of point (1) of Theorem~\ref{th.entropyNTSF} which is inspired by the proof of Theorem~5.1 of \cite{jager2021onset}.

\begin{proof}[Proof of point (1) of Theorem 3]

Let us assume that for a contradiction that some strictly positive parameters $a$ and $b$ the map $f_{a,b}$ has no topological entropy, and for simplicity we denote $f_{a,b}$ just $f$. Of course, from the considerations of the previous subsection, that the map has bounded diffusion and that the orbit of any point $z$ remains in an horizontal strip of uniformly bounded width. This, and the fact that $f$ preserves area, implies that the non-wandering set of $f$ is the whole annulus $\A$. Furthermore, since the full orbit of every point is pre-compact, and since we assumed for a contradiction that $f$ has no topological horseshoe, we obtain from Theorem~A of \cite{lecalveztal} that the rotation number of every point is well defined and that the function $\rho:\A\to\R$ that, for each point assigns its rotation number is continuous in $\A$. One also sees that, since $a>0$, that each connected component of a level set of $\rho$ must be a compact annular set, and as such it must contain a circloid.

We also note that the only fixed points of $f$ with zero rotation number are the points $\{(0,-1), (0,1), (1/2,-1), (1/2,1)\}$. A simple inspection shows that neither $1$ nor $-1$ are eigenvalues of the differential of $f$ at any of these points, meaning that $f$ has no degenerate fixed points. But $\rho^{-1}(0)$ must contain a circloid $\mathcal{C}$. The connected component  of $\rho^{-1}((-1, 1))$ that contains $\mathcal{C}$ is a bounded invariant open topological annulus, and the restriction of $f$ to this annulus is an area-preserving map. Now, since every point in $\mathcal{C}$ has null rotation number, we can apply Theorem~1.4 from \cite{Koromeysam} which shows that either $\mathcal{C}$ has degenerate fixed points, or every point in $\mathcal{C}$ is either a saddle or contained in an heteroclinic connection between saddles. Since the former cannot be, one deduces that $f$ has saddle connections. But $f$ extends to a biholomorphic map of $\mathbb{C}^2$, and Ushiki's Theorem (see, for instance, \cite{hasselblatt2002handbook}, p.289) says that biholomorphisms have no saddle connections between hyperbolic saddles, a final contradiction. 
\end{proof}

\section{CAP: Chaos for conservative maps}

We will consider the following variations of the Standard Family (S.F.). 
The S.F.~is given by those maps having lifts
$$f_{a,b}(x,y)=(x+a\,y\,,\,y+b\sin(2\pi\,(x+a\,y))),\ a\in\R,$$ 
and the variations we are considering are given by those maps having lifts 
$$f_{\textbf{h},\textbf{v}}(x,y)=(x+\textbf{h}(y)\,,\,y+\textbf{v}\,(\sin(2\pi\,(x+\textbf{h}(y)))),$$
where
\begin{itemize}
\item $h\in C^{0}(\R),\, [-1,1]\subset h([-1,1])$,
\item $v\in C^{0}(\R),\ \int_{[0,1]}v(\sin(2\pi\,x))\,dx=0$,
\item $h$ lifts a map of $\T^1$.
\end{itemize}
It turns out that these maps contain all the possible analytic diffeomorphisms on $\A$ which lift
maps of $\T^2$ having 0 mean
vertical displacement. We write them in this way
so that they can be regarded as variations of elements of the S.F. 
Moreover, depending on
$h$ one can have a twist map (e.g.~$h(y)=y$) or
a non-twist map (e.g.~$h(y)=\sin(2\pi\,y)$) which lifts a torus homeomorphism in the homotopy class
of the identity.

\smallskip

We want to apply Theorem \ref{thm:thm1} to the twist case, and Theorem \ref{thm:thm2} 
to the non-twist case, meaning we need the following information:

\begin{itemize}

\item Twist case: the number $N(f)=\max_{x\in[-1,1]}|v(x)|$ which will be smaller than $2$ 
for all the examples considered below. 
This means that we have to consider $M=6$ in Theorem \ref{thm:thm1}.

\item Non-twist case: $N(f)=\max_{x\in[-1,1]}|v(x)|$ and the rotational difference $\rho$. 
For all the examples considered below, we have $N(f)\leq 2$ and $\rho=2$. This means that 
we have to consider $M_2=10$ in Theorem \ref{thm:thm2}.

\end{itemize}

Thus, for all the examples considered in the following, in case we can find 
a point in the projection of $y=-5$ which reaches the region above $y=5$, 
then we obtain the existence of total unbounded diffusion and a rotational horseshoe. 
%\marginpar{Maciej: Why \color{red}$y=\pm 5$? \color{black} Later we have $3$ and $7$}
%\marginpar{Alejandro: This is fixed now, there was a confusion about which result we wanting 
%to apply}

\subsection{CAP-setup}
\label{sec:parallel-shooting}In this section we will show a method, which
ensures that there exists a trajectory which starts at $y=-B$ and finishes
above $y=B$. (In the conservative map case we will consider $B=5$.)

A first approach would be to consider a single point at $y=-B$, and use
interval arithmetic to compute bounds on its iterates until it reaches the
domain above $y=B$. Such approach will be adequate when a map does not have
a direction of strong hyperbolic expansion and when the number of required
iterates is small. We will want to use a method which allows us to take many
iterates, so that we do not need to worry about a potential `blowup'. 
For instance, we would like our method to be applicable, say to the case where we have an expansion of order $2$ (accompanied with contraction of order $\frac{1}{2}$ in a
complementary hyperbolic direction), and the number of iterates is of
order $30$. The overall expansion for one hundred iterates of the map would
be $2^{30}.$ In such a case direct interval arithmetic iterations of the map
have no chance of success. (In Section \ref{sec:ntsf} we will in fact require up to $300$ iterates of a map.)

Our method is based on the parallel shooting approach. For fixed chosen
vectors $q_{0},q_{1},v_{0},v_{1}\in\mathbb{R}^{2}$ we will find two numbers
$h_{0},h_{1}$ and a trajectory from $q_{0}+h_{0}v_{0}$ to $q_{1}+h_{1}v_{1}$.
In more detail, we will estasblish a trajectory $\{p_{i}\}_{i=0}^{m}$%
\begin{align}
p_{0} &  =q_{0}+h_{0}v_{0},\nonumber\\
p_{k} &  =f\left(  p_{k-1}\right)  \qquad\text{for }k=1,\ldots
,m,\label{eq:trajectory}\\
p_{m} &  =q_{1}+h_{1}v_{1}.\nonumber
\end{align}
For our application we will position $q_{0}$ on $y=-B$, take $v_{0}=\left(
1,0\right)  $, which ensures that $p_{0}=q_{0}+h_{0}v_{0}$ is on the line $y=-B$ and consider a
sufficiently large $m$ so that we will expect to go above $y=B$ in $m$
iterates of the map. We have a degree of freedom of the choice of $q_0, q_{1},v_{1}$ and $m$, but $q_{1}$ should be chosen close to $f^{m}\left(  q_{0}\right)  $.

The fact that we choose the initial and the final point of the form
$p_{0}=q_{0}+h_{0}v_{0}$ and $p_{m}=q_{1}+h_{1}v_{1}$ might seem artificial at
first glance. The reason for doing so is to reformulate our problem to a setting which ensures
that there can be only one solution for (\ref{eq:trajectory}). This allows us
to cast the problem in the form of a Newton-type method. Namely, for the fixed
$q_{0},q_{1},v_{0},v_{1}\in\mathbb{R}^{2}$ we consider a function
\[
F:\mathbb{R}\times\mathbb{R}\times\underset{m-1}{\underbrace{\mathbb{R}%
^{2}\times\ldots\times\mathbb{R}^{2}}}=\mathbb{R}%
^{2m}\rightarrow\mathbb{R}^{2m}%
\]
defined as%
\begin{equation}
F\left(  h_{0},h_{1},p_{1},\ldots,p_{m-1}\right)  =\left(
\begin{array}
[c]{l}%
f\left(  p_{1}\right)  -p_{2}\\
f(p_{2})-p_{3}\\
\qquad\vdots\\
f\left(  p_{m-2}\right)  -p_{m-1}\\
f\left(  p_{m-1}\right)  -\left(  q_{1}+h_{1}v_{1}\right)  \\
f\left(  q_{0}+h_{0}v_{0}\right)  -p_{1}%
\end{array}
\right)  . \label{eq:F-Krawczyk}
\end{equation}
Establishing that $F\left(  h_{0},h_{1},p_{1},\ldots,p_{m-1}\right)  =0$ is
equivalent to finding the trajectory (\ref{eq:trajectory}).

Validating in interval arithmetic that there exists a solution of $F=0$ can be
performed by means of the Krawczyk method. We first state the theorem and afterwards comment how it applies in our setting.

\begin{thm}
\label{th:krawczyk}\cite{Alfred}(Krawczyk method)
Let $\left[
X\right]  \subset\mathbb{R}^{n}$ be an interval set (i.e. a product of
intervals), let $X_{0}\in\left[  X\right]  $ and let $C\in\mathbb{R}^{n\times
n}$ be a linear isomorphism. Let $F:\mathbb{R}^{n}\rightarrow\mathbb{R}^{n}$
be a $C^{1}$ map and let $\left[  DF\left(  \left[  X\right]  \right)
\right]  \subset\mathbb{R}^{n\times n}$ stand for the interval enclosure of
the map%
\[
\left[  DF\left(  \left[  X\right]  \right)  \right]  :=\left\{  \left(
a_{ij}\right)  _{i,j=1,...,n}:a_{ij}\in\lbrack\inf_{p\in\left[  X\right]
}\frac{\partial F_{i}}{\partial x_{j}}\left(  p\right)  ,\sup_{p\in\left[
X\right]  }\frac{\partial F_{i}}{\partial x_{j}}\left(  p\right)  ]\right\}  .
\]
 Let
\[
K:=X_{0}-CF\left(  X_{0}\right)  +\left(  Id-C\left[  DF\left(  \left[
X\right]  \right)  \right]  \right)  \left(  \left[  X\right]  -X_{0}\right)
.
\]
If $K\subset\mathrm{int}\left[  X\right]  $ then there exists in $\left[
X\right]  $ exactly one solution of $F\left(  x\right)  =0.$ 
\end{thm}

To apply the Krawczyk method to establish (\ref{eq:trajectory}) we proceed as
follows. We first use standard numerical computations\footnote{This does not
involve interval arithmetic, meaning that all results should be treated as
approximations.} to find $q_{0},q_{1}$ and $m$ such that $q_{0}\in\left\{
y=-B\right\}  $, $q_{1}\approx f^{m}\left(  q_{0}\right)  $ and $q_{1}%
\in\left\{  y>B\right\}  $. This also gives us an approximation of the
trajectory we seek as $\hat{p}_{i}\approx f^{i}\left(  q_{0}\right)  $ for
$i=0,\ldots,m-1$. (We emphasise here that all non-interval numerical computations involve rounding errors, so a trajectory computed in such way will never be fully precise; hence we highlight this writing `$\approx$' instead of `$=$'.) We choose $v_{0}=\left(  1,0\right)  $, which ensures that
$q_{0}+h_{0}v_{0}\in\left\{  y=-B\right\}  $ for every $h_{0}\in\mathbb{R}$,
and select an arbitrary\footnote{Our first choice of $v_{1}$ can be
$v_{1}=\left(  1,0\right)  $ which automatically ensures that if $q_{1}\in \left\{  y>B\right\}$ then also $q_{1}%
+h_{1}v_{1}\in\left\{  y>B\right\}  $ for every $h_{1}\in\mathbb{R}$. It could
be that with such choice our validation is not succesful, in which case we can
try again with $v_{1}=\left(  0,1\right)  $.} vector $v_{1}$. We then choose a
point $\hat{X}\in\mathbb{R}^{2m}$, $\hat{X}:=\left(  0,0,\hat{p}_{0}%
,\ldots,\hat{p}_{m-1}\right)  $ and choose $C\approx(DF(\hat{X}))^{-1}.$ All
this is done using standard (not interval arithmetic) numerics.

We then proceed to the validation by enlarging $\hat{X}$ on all its
coordinates to obtain a box $\left[  X\right]  .$ We choose $X_{0}=\hat{X}$ and
validate that inside of $\left[  X\right]  $ we have the solution of $F=0$ by
using Theorem \ref{th:krawczyk}, performing all computations in interval
arithmetic. We also need to check that the resulting bound on $p_{m}$ is above
$y=B$. We make use of the fact that the bound for $h_{1}$ for which we have
(\ref{eq:trajectory}) is the second coefficient of the box $\left[  X\right]
$, i.e. $\left[  X\right]  _{2}$. We therefore check in interval arithmetic
that $p_m=q_{1}+\left[  X\right]  _{2}v_{1}\in\left\{  y>B\right\}  .$

%With this approach we validate the existence of a point in the projection of
%$y=-5$ which reaches the region above $y=5$, which implies the existence of
%the rotational horseshoe.

\subsection{Results} 

Tables \ref{tb:tableconservative1} and \ref{tb:tableconservative2}  show several situations where the procedure described above gave a computer-assisted proof of the existence of diffusion (and therefore also chaos and non-empty rotation set). 
It deals with variations of the Standard Family, and its purpose is also to show the flexibility of the method. 
The tables list the specific maps that were studied as well as the number of iterates needed to find an orbit which goes from the line $y=-5$ above the line $y=5$. 
Table \ref{tb:tableconservative1} deals with twist maps  while Table \ref{tb:tableconservative2} deals with the non-twist cases. See also Figure \ref{fig:conservative}, where the successive iterates have been joined by lines to highlight which point is mapped into which.

\begin{table}
{\scriptsize
\begin{center}
\begin{tabular}{l l l l l} 
\hline 
	&  	$\textbf{v}(x)$	& $\mathbf{c}$	& 	initial point & it. \\
\hline
\hline
1	 	&	$x+c$		& $0$	& 	$(0.2647+2.95\cdot 10^{-14}\cdot [-1, 1],-5)$ & 11 \\
2		& $x\,(1-x)+c$ 	& $0.5$ 	& 	$(0.3695+1.02\cdot 10^{-14}\cdot[-1,1],-5)$ & 15 \\
3	 	&  $\tan(x)	+c$ & $0$	& 	$(0.5266+5.45\cdot 10^{-15}\cdot[-1, 1],-5)$  & 12 \\
4	 	&  $3\,\ln(x+2)+c$ & $[-1.8714651070, -1.8713991910]$	& 	$(0.4387+1.25\cdot10^{-5\phantom{4}}\cdot[-1, 1], -5)$  & 10 \\
5	 	&  $e^x-1+c$ & $[-0.2660893818, -0.2660423737]$	& 	$ (0.3693+7.02\cdot 10^{-6\phantom{4}}\cdot[-1, 1],-5)$  & 12 \\
\hline
\end{tabular}
\end{center}}
\caption{The twist map case with $\textbf{h}(y)=y$. \label{tb:tableconservative1}}
\end{table}

\begin{table}
{\scriptsize
\begin{center}
\begin{tabular}{l l l l l} 
\hline 
	&  	$\textbf{v}(x)$	& $\mathbf{c}$	& 	initial point & it. \\
\hline
\hline
1	 	&	$x+c$		& $0$	& 	$(0.4454+8.22\cdot10^{-15}\cdot[-1,1],-5) $ & 11 \\
2		& $x\,(1-x)+c$ 	& $0.5$ 	& 	$(0.4185+3.04\cdot 10^{-13}\cdot[-1, 1],-5)$ & 14 \\
3	 	&  $\tan(x)	+c$ & $0$	& 	$(0.3332+1.12\cdot 10^{-14}\cdot [-1, 1],-5)$  & 7 \\
4	 	&  $3\,\ln(x+2)+c$ & $[-1.8714651070, -1.8713991910]$	& 	$(0.4614+8.45\cdot 10^{-6\phantom{4}}\cdot[-1,1],-5)$  & 9 \\
5	 	&  $e^x-1+c$ & $[-0.2660893818, -0.2660423737]$	& 	$ (0.3046+1.25\cdot 10^{-5\phantom{4}}\cdot[-1,1],-5)$  & 8\\
\hline
\end{tabular}
\end{center}}
\caption{The non-twist map case with $\textbf{h}(y)=\sin(2\pi\,y)$. \label{tb:tableconservative2}}
\end{table}

\begin{figure}
\begin{center}
\includegraphics[width=6cm]{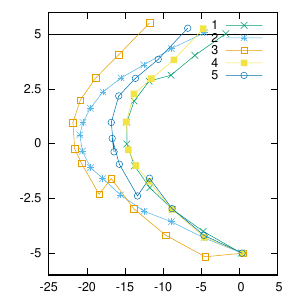}\includegraphics[width=6cm]{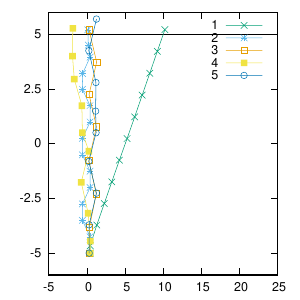}
\end{center}
\caption{The trajectories for the twist maps from Table \ref {tb:tableconservative1} (left), and the trajectories for the non-twist maps from Table \ref{tb:tableconservative2} (right). The plots are for the lift of the map.\label{fig:conservative}}
\end{figure}

The constants $c$ for the maps $\textbf{v}(x)$ needs to be chosen so that 
\begin{equation}
\int_0^1\textbf{v}(\sin(2\pi\,x))\,dx=0.\label{eq:c-choice}
\end{equation}  In some cases such integral is not straightforward to compute. Then we can take $\mathbf{w}(x):=\mathbf{v}(\sin(2\pi\,x))-c$ and choose an interval $\mathbf{c}\subset\mathbb{R}$ so that $-\int_0^1 \textbf{w}(x)dx\in \mathbf{c}$, and validate the existence of the required trajectories for all maps for that interval $\mathbf{c}$. This in particular means that the bounds are valid for the particular value $c$ for which we have (\ref{eq:c-choice}). Obtaining an interval arithmetic bound on an integral of a function $\textbf{w}:[0,1]\to\mathbb{R}$ is straightforward. We can choose $0=x_0<x_1<\ldots<x_k=1$ and compute the following bound in interval arithmetic
\[\int_0^1 \textbf{w}(x)dx \in \sum_{i=1}^k [\textbf{w}([x_{i-1},x_{i}])](x_{i}-x_{i-1}).\]The bound is valid since the interval on the right hand side contains the upper and the lower Riemann sums for the partition $\{x_i\}_{i=0,\ldots,k}$.

In Tables \ref{tb:tableconservative1} and \ref{tb:tableconservative2} we include an initial point for each orbit. Note that it is in a form of an interval. This is because we display here the end result of the validation performed using the parallel shooting approach from section \ref{sec:parallel-shooting}. The method produces an interval enclosure of the validated trajectory. We display only the bound for the first point; the remaining points are validated within an enclosure of similar size. We can see that the enclosure is larger in the case where we have a rough bound for $c\in \mathbf{c}$, which is to be expected.

We have used the above method to provide a proof of Theorem \ref{thm:cap-vsf}. The code for the proof is available in \cite{code}.

\section{CAP: Unbounded diffusion in the Non-Twist Standard Family\label{sec:ntsf}}

The aim of this section is to provide an application of Theorem \ref{thm:thm3} to prove Theorem \ref{thm:cap-ntsf}.

\subsection{CAP-setup}

Applying Theorem \ref{thm:thm3} follows the same method as in Section \ref{sec:parallel-shooting}. Namely, for a chosen parameter pair $(a,b)$ we choose $B>M_{a,b}$, where $M_{a,b}$ is given by (\ref{eq:Mab-ntsf}), and validate the existence of a trajectory which starts at $y=-B$ and goes above $y=B$.

\subsection{Results}
We have chosen a mesh of $1000\times 1000$ uniformly distributed parameter points $(a,b)\in [0,1]^2$ and have checked whether the assumptions of Theorem \ref{thm:thm3} hold for each point. To do so, for each point, we first perform a non-rigorous numerical investigation to find a trajectory from $y=-B$, for $B>M_{a,b}$, which requires the smallest number of iterates needed to go above $y=B$. (The longest trajectory we have allowed for is to have $300$ iterates.)

If for a chosen parameter pair we do not find such trajectory, we consider our validation to have failed and proceed to another parameter pair. If we do find a non-rigorous candidate for a trajectory which goes from $y=-B$ to above $B$, then we validate it by showing that we have $F=0$ for $F$ defined by (\ref{eq:F-Krawczyk}), by means of the Krawczyk method from Theorem \ref{th:krawczyk}. If we fail to validate the trajectory by means of the computer-assisted proof, then we proceed to another parameter pair; if the validation is successful though, then we have placed a red dot for the parameter pair in Figure \ref{fig:ntsf}.

\begin{remark}
In principle, one could consider parameter boxes $A\times B \subset [0,1]^2$ and validate assumptions of Theorem \ref{thm:thm3} on them, instead of considering single parameter pairs $(a,b)$. Such approach is not feasible though from a practical point of view, and we were not able to cover the entire box $[0,1]^2$ to validate the entire area. The problem is that many regions of the parameter domain require large numbers of iterates; we consider up to $300$ iterates. If we choose a parameter box $A\times B$ then on each iterate of the map we accumulate an error, which  leads to a blowup unless the box is tiny; say a product of intervals of length $10^{-6}$. Considering $10^{-6}\times 10^{-6}$ boxes was not feasible for us\footnote{Producing Figure \ref{fig:ntsf} by investigating $1000\times 1000$ mesh points required close to 5 hours of computation on a cluster, on which we performed the taks on $48$ parallel threads. (The computation time was under $5\cdot 48$ hours on a single thread.)}.
\end{remark}

\section{CAP: Chaos in the Dissipative Standard Family}\label{s.dsfcap}

The Disipative Standard Family is given by 
\begin{align*}
    (x,y)\mapsto (x+a\,y\,\,,\,\,b\,y+\sin(2\pi\,(x+a\,y)))
\end{align*}
with $0<b<1$ being the determinant of the Jacobian of the diffeomorphism. 
In our setting, we assume $a>2$ to ensure that there are fixed points with distinct rotation vectors. In this section we provide the needed tools and also outline the proof of Theorem \ref{thm:cap-dsf}.

\subsection{Preliminaries for applying Corollary B}

For this section, the adapted version of Corollary B reads as follow: 
Note that in the following $\left[ r \right]$ refers to the upper integer part of the number $r$, that is, its the usual integer part plus 1.

\begin{corB}
 Assume that for $B>0$ we have that the line $y=B$ is sent bellow it self and the line $y=-B$ is sent above itself under the map above.
 Further, take a pair of points $(x_\kappa,y_\kappa)$ and $(x_{-\kappa},y_{-\kappa})$ such that 
 $(x_\kappa,y_\kappa)$ is sent to $(x_\kappa+\kappa,y_\kappa)$ and $(x_{-\kappa},y_{-\kappa})$ to $(x_{-\kappa}-\kappa,y_{-\kappa})$ for some integer $\kappa\neq 0$.
 Then, if there exists a $\left[\frac{34}{2|\kappa|}\right]$-Disjoint
 Pair of Neighborhoods $U_0,U_1$, for the projection of $(x_\kappa,y_\kappa)$ and $(x_{-\kappa},y_{-\kappa})$ in the annulus, whose backward orbits meet the projection of both lines $y=B$ and $y=-B$, then the induced annular map has a Rotational Horseshoe.
\end{corB}

In the following, we set up the necessary tools required to apply this result.
For doing this, we will actually work with the inverse map of the one defined above, which is given by
\begin{align*}
    f_{a,b}(x,y)=\left(x-\frac{a}{b}\left(y-\sin(2\pi\,x)\right),\frac{1}{b}(y-\sin(2\pi\,x))\right).
\end{align*}

\subsubsection{Pair of fixed points}

We look for a fixed point (in the annulus) which rotates with a integer $\kappa$ for the map
$f_{a,b}$, which means that it is a point whose
lift $(x_\kappa,y_\kappa)$ verifies
$$f_{a,b}(x_\kappa,y_\kappa)=(x_\kappa+\kappa,y_\kappa),$$
with $\kappa\in\mathbb N$. 
We can find such a point in the plane in the following way:
\begin{equation}
    \begin{cases}
      x-\frac{a}{b}(y-\sin(2\pi\,x))=x+\kappa\\
      \frac{1}{b}(y-\sin(2\pi\,x))=y
    \end{cases}\
\end{equation}
which can be solved, whenever
\begin{align}\label{cond:fix_point_existence}
    \left | \frac{\kappa}{a}(b-1) \right |<1,
\end{align}
%\marginpar{Maciej: I believe (\ref{cond:fix_point_existence}) should be $| \frac{\kappa}{a}(b-1)|<1$}
by the following point
\begin{equation}
    \begin{cases}
      x_\kappa= \frac{\arcsin\left(\frac{\kappa}{a}(b-1)\right)}{2\pi}\\
      y_\kappa=\frac{-\kappa}{a}
    \end{cases}.
\end{equation}

In order to apply Corollary B we need to work with pairs of fixed points. 
We will consider pairs of fixed points given by $(x_{\kappa},y_{\kappa})$ and
$(x_{-\kappa},y_{-\kappa})$ which have a rotational difference of $2|\kappa|$.

\begin{remark}
Note that we could take several pairs of fixed points, namely all pairs where $\kappa$ satisfies the condition \eqref{cond:fix_point_existence}.
However, there is no clear favorite pair in order to apply
the corollary: if $|\kappa|$ is increased, the required Disjoint Pair of Neighborhoods (DPN) is less demanding, but the distance from the fixed point $(x_\kappa,y_\kappa)$ to the line $y=-B$ gets larger.
If we decrease $|\kappa|$, we get a more demanding DPN but a smaller distance to the line $y=-B$.
\end{remark}

\subsubsection{Line below and line above}

We need to determine the lines $y=B$ and $y=-B$ mentioned in the corollary. This can be easily done as follows: we need that
$$b\,B+\sin(2\pi\,x)\leq B$$
for which it is enough that

$$b\,B+1\leq B\,,\mbox{ so  }B\geq\frac{1}{1-b}$$
Thus, one can take $B=\frac{1}{1-b}$.

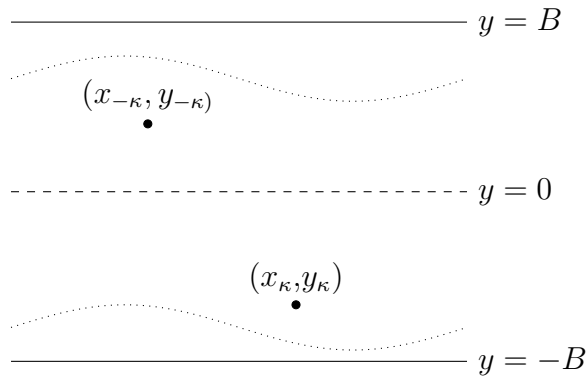
\begin{figure}[h]
    \centering

    \begin{tikzpicture}[scale=1.5]
        % Top Line
        \draw (0,3) -- (4,3) node[right] {$y = B$};
%        \draw[dotted] plot[smooth, tension=1.2] coordinates {(0,2.7) (0.5,2.5) (1,2.6) (1.5,2.6) (2,2.3) (2.5,2.8) (3,2.7) (4,2.9)};
        \draw[dotted] plot[domain=0:4, samples=100] (\x,{2.5+0.2*sin(360*\x/4)}) node[right] {};
        \filldraw (1.2,2.1) circle (1pt) node[above] {$(x_{-\kappa}, y_{-\kappa)}$};
    
        % Middle Line
        \draw[dashed] (0,1.5) -- (4,1.5) node[right] {$y = 0$};
    
        % Bottom Line
        \draw (0,0) -- (4,0) node[right] {$y = -B$};
%        \draw[dotted] plot[smooth, tension=1.2] coordinates {(0,0.3) (0.5,0.3) (1,0.7) (1.5,0.2) (2,0.4) (2.5,0.2) (3,0.2) (4,0.4)};
        \draw[dotted] plot[domain=0:4, samples=100] (\x,{0.3+0.2*sin(360*\x/4)}) node[right] {};
        \filldraw (2.5,0.5) circle (1pt) node[above] {$(x_\kappa$,$y_\kappa)$};
    \end{tikzpicture}    

    \caption{Sketch of the configuration of fixed points as well as of the lines $y=B$ and $y=-B$ together with their forward images (dotted lines).}
    \label{fig:enter-label}
\end{figure}

\subsection{Applying Corollary B\label{sec:Cor-B-appl-1}}

In order to successfully apply the statement, first we choose a pair of fixed points
$$(x_{\kappa},y_{\kappa})\mbox{  and  }(x_{-\kappa},y_{-\kappa})$$ with
$\kappa\neq 0$ in $\mathbb{N}$. Thus, following Corollary $B$ we need
 to find a $\left[\frac{34}{2\,\kappa}\right]$-DPN $U_0,U_1$ for the mentioned
pair of fixed points whose backward orbits reach the projection of both lines $y=-B$
and $y=B$. For doing so, it is enough to find four segments
$$S_0^+\,,\,S_1^+\,,\,S_0^-\,,\,S_1^-\subset\mathbb R^2$$
verifying the following:

\begin{enumerate}
\item[(I)] $S_0^+$ starts at $(x_\kappa,y_\kappa)$ and ends at a point $(x^+_0,y^+_0)$ which goes above the line $y=B$ after \textbf{some number} of steps.
First, one needs to find such a point $(x^+_0,y^+_0)$ and then work with the corresponding segment for which the next condition must be fulfilled.

\item[(II)] For $\ell=0,1,\ldots,N=\left[\frac{34}{2\kappa}\right]$ it holds that
$$f_{a,b}^\ell(S_0^+)\subset B^{\kappa}_\ell$$  
where $B^{\kappa}_\ell$ is the box centered at $(x_\kappa+\kappa\,\ell,y_\kappa)$ having width 1 and height 2 (recall that $\kappa$ is positive), see Figure \ref{fig:pos_Bu_Ks}.
\end{enumerate}

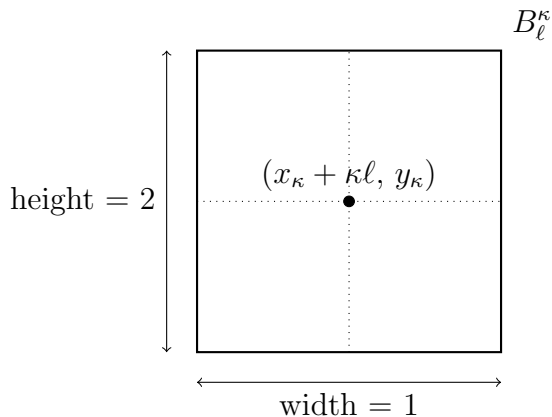
\begin{figure}[h]
    \centering

    \begin{tikzpicture}[scale=2]
        % Draw square
        \draw[thick] (0,0) rectangle (2,2);
    
        % Draw cross
        \draw[dotted] (1,0) -- (1,2);
        \draw[dotted] (0,1) -- (2,1);
    
        % Draw labels
        \node[above right] at (2,2) {$B^{\kappa}_\ell$};
    
        % Draw point
        \filldraw (1,1) circle (1pt) node[align=center, above] {$(x_\kappa+\kappa \ell$, $ y_\kappa)$};
    
        % Optical indicators
        \draw[<->] (0,-0.2) -- (2,-0.2) node[midway, below] {width = 1};
        \draw[<->] (-0.2,0) -- (-0.2,2) node[midway, left] {height = 2};
    \end{tikzpicture}

    \caption{Sketch of the boxes $B^{\kappa}_\ell$ for the iterations of $(x_\kappa,y_\kappa)$.}
    \label{fig:pos_Bu_Ks}
\end{figure}

\begin{enumerate}
\item[(III)] $S_1^+$ starts at $(x_{-\kappa}$, $y_{-\kappa})$ and ends at a point $(x^+_1,y^+_1)$ which goes above the line $y=B$ after \textbf{some number} of steps. 
    Again, one needs to find such a point $(x^+_1,y^+_1)$ first and then work with the corresponding segment for which the next condition must be fulfilled.

\item[(IV)] For $\ell=0,1,\ldots,N=\left[\frac{34}{2\kappa}\right]$ it holds that $$f_{a,b}^\ell(S_1^+)\subset B^{-\kappa}_\ell$$ with $B^{-\kappa}_\ell$ the box centered at $(x_{-\kappa}-(\kappa\ell), y_{-\kappa})$ with width $1$ and height $2$, see Figure \ref{fig:pos_Bl_Ks}.
\end{enumerate}

\begin{figure}[h]
    \centering

    \begin{tikzpicture}[scale=2]
        % Draw square
        \draw[thick] (0,0) rectangle (2,2);
    
        % Draw cross
        \draw[dotted] (1,0) -- (1,2);
        \draw[dotted] (0,1) -- (2,1);
    
        % Draw labels
        \node[above right] at (2,2) {$B^{-\kappa}_\ell$};
    
        % Draw point
        \filldraw (1,1) circle (1pt) node[align=center, above] {$(x_{-\kappa}-(\kappa\ell)$, $y_{-\kappa})$};
    
        % Optical indicators
        \draw[<->] (0,-0.2) -- (2,-0.2) node[midway, below] {width = 1};
        \draw[<->] (-0.2,0) -- (-0.2,2) node[midway, left] {height = 2};
    \end{tikzpicture}

    \caption{Sketch of the boxes $B^{-\kappa}_\ell$ for the iterations of $(x_{-\kappa}$, $y_{-\kappa})$.}
    \label{fig:pos_Bl_Ks}
\end{figure}
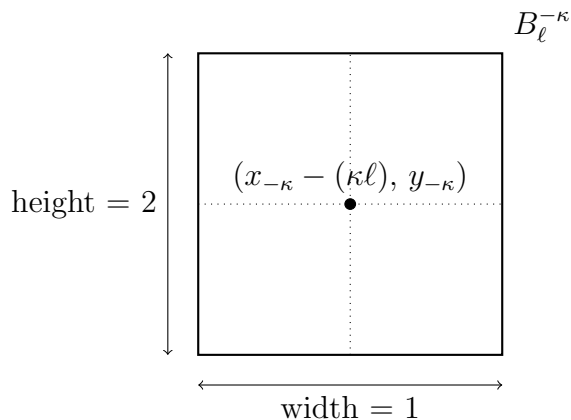

The next 4 conditions for the segments $S_0^-$ and $S_1^-$, which are needed, are just symmetric to those in (1)-(4) but this time we need to reach the line $y=-B$.

\begin{enumerate}

\item[(V)]$S_0^-$ starts at $(x_\kappa,y_\kappa)$ and finishes at a point $(x^-_0,y^-_0)$ which goes below the line $y=-B$ after \textbf{some number} of steps.
    First, one needs to find such a point $(x^-_0,y^-_0)$ and then work with the corresponding segment for which the next condition must be fulfilled.

\item[(VI)] For $\ell=0,1,\ldots,N=\left[\frac{34}{2\kappa}\right]$ it holds that $$f_{a,b}^\ell(S_0^-)\subset B^{\kappa}_\ell.$$

\item[(VII)] $S_1^-$ starts at $(x_{-\kappa}$, $y_{-\kappa})$ and finishes at a point $(x^-_1,y^-_1)$ which goes below the line $y=-B$ after \textbf{some number} of steps. 
    Again, one needs to find such a point $(x^-_1,y^-_1)$ first and then work with the corresponding segment for which the next condition must be fulfilled.

\item[(VIII)] For $\ell=0,1,\ldots,N=\left[\frac{34}{2\kappa}\right]$ it holds that $$f_{a,b}^\ell(S_1^-)\subset B^{-\kappa}_\ell.$$
\end{enumerate}

\subsection{CAP-setup}

In this section we address the issue of how to validate the conditions (I)-(VIII). The direct
validation of these conditions can be problematic in the setting where the map
$f_{a,b}$ exhibits strong expansion and the number $N$ is large. To address this issue we consider a more careful approach than by direct iteration of sets in interval arithmetic.

Our approach consists of two steps. The first step is to ensure that we can choose our segments arbitrarily close to the fixed points, and that we can reach a prescribed curve by iterating a point from such arbitrarily small segments. The second step is to ensure that points from the chosen curve go above or below the chosen line.

\subsubsection{Choice of segments}
We start by introducing some notation.
We consider two invertible matrices $A_{0},A_{1}\in\mathbb{R}^{2\times2}$ and
for $L,\delta>0$ we will consider segments defined as
\begin{align*}
S_{0}^{\pm}\left(  \delta\right)   &  :=\left\{  \left(  x_{\kappa},y_{\kappa
}\right)  +A_{0}\left(  x,y\right)  :\left\vert y\right\vert \leq L\left\vert
x\right\vert ,\left\vert x\right\vert \leq\delta\right\}  ,\\
S_{1}^{\pm}\left(  \delta\right)   &  :=\left\{  \left(  x_{-\kappa
},y_{-\kappa}\right)  +A_{1}\left(  x,y\right)  :\left\vert y\right\vert \leq
L\left\vert x\right\vert ,\left\vert x\right\vert \leq\delta\right\}  .
\end{align*}
We also define
\[
\tilde{S}_{0}^{\pm}\left(  \delta\right)  :=A_{0}^{-1}\left(  S_{0}^{\pm
}\left(  \delta\right)  -\left(  x_{\kappa},y_{\kappa}\right)  \right)
,\qquad\tilde{S}_{1}^{\pm}\left(  \delta\right)  :=A_{1}^{-1}\left(
S_{1}^{\pm}\left(  \delta\right)  -\left(  x_{-\kappa},y_{-\kappa}\right)
\right)  .
\]
We can think of $\tilde{S}_{0}^{\pm}$ and $\tilde{S}_{1}^{\pm}$ as segments
expressed in the local coordinates given by $A_{0}$ and $A_{1}$ at
$(x_{-\kappa},y_{-\kappa})$ and $(x_{\kappa},y_{\kappa})$, respectively. See
Figure \ref{fig:segments}. (Note that in the local coordinates all the segments will have
the same shape.)

\begin{figure}[ptb]
\begin{center}
\includegraphics[width=12cm]{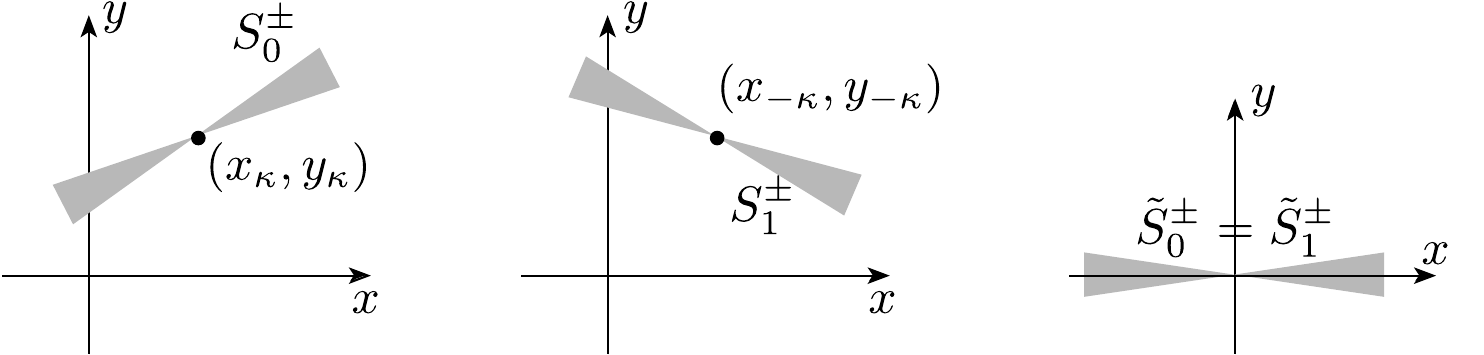}
\end{center}
\caption{In the left and middle plots we have the segments in the original
coordinates, in the right plot we see the segments in the local coordinates at
our fixed points. }%
\label{fig:segments}%
\end{figure}

Our objective will be to be able to choose segments $S_{i}^{\pm}(\delta)$ for
$i=0,1,$ with arbitrarily small $\delta$, making them as close to the fixed
points as we want. At the same time we will want to ensure that for
sufficiently high iterats we can establish that there exist points in these
segments which go above or below a prescribed line.

Let us start with a simple fact.

\begin{lemma}
\label{lem:small-delta}For every $N>0$ there exists a $\delta$ such that
\[
f_{a,b}^{\ell}\left(  S_{0}^{\pm}\left(  \delta\right)  \right)  \subset
B_{\ell}^{\kappa}\qquad\text{and\qquad}f_{a,b}^{\ell}\left(  S_{1}^{\pm
}\left(  \delta\right)  \right)  \subset B_{\ell}^{-\kappa}\qquad\text{for
}\ell=1,\ldots,N.
\]

\end{lemma}

\begin{proof}
This follows from the continuity of $f_{a,b}$.
\end{proof}

The above lemma means that the conditions (II), (IV), (VI) and (VIII) will be
automatically satisfied if we choose sufficiently small $\delta$. This on its
own is will not be satisfactory, since we also need to ensure conditions (I),
(III), (V) and (VII), namely that we can go above the line $y=B$ and below the line
$y=-B$ from these segments. This is an issue we will need to address.

First we need some additional notation. Let us introduce the following two
functions $f_{0},f_{1}:\mathbb{R}^{2}\to\mathbb{R}^{2}$, defined as
\begin{align*}
f_{0}\left(  x,y\right)   &  :=A_{0}^{-1}\left(  f_{a,b}\left(  \left(
x_{\kappa},y_{\kappa}\right)  +A_{0}\left(  x,y\right)  \right)  -\left(
x_{\kappa},y_{\kappa}\right)  \right)  ,\\
f_{1}\left(  x,y\right)   &  :=A_{1}^{-1}\left(  f_{a,b}\left(  \left(
x_{-\kappa},y_{-\kappa}\right)  +A_{1}\left(  x,y\right)  \right)  -\left(
x_{-\kappa},y_{-\kappa}\right)  \right)  .
\end{align*}
These maps express the map $f_{a,b}$ in the coordinates given by $A_{0}$
and $A_{1}$ at $(x_{-\kappa},y_{-\kappa})$ and $(x_{\kappa},y_{\kappa})$, respectively.

\begin{figure}[ptb]
\begin{center}
\includegraphics[width=11cm]{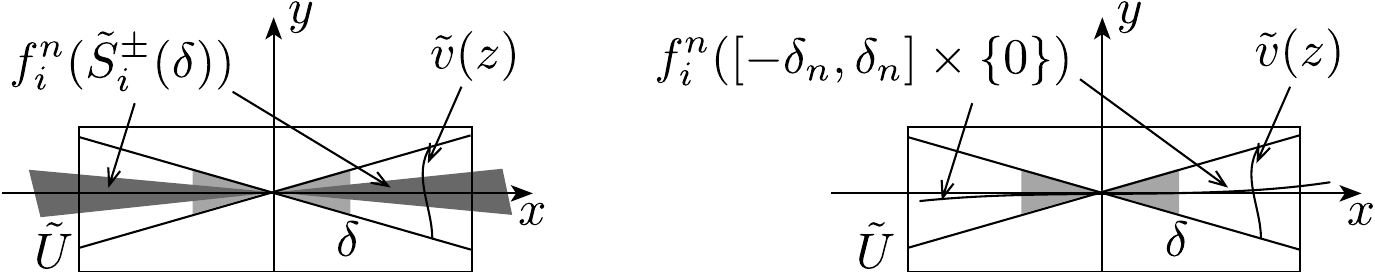}
\end{center}
\caption{An illustration for Lemma \ref{lem:exit-points}. }%
\label{fig:exit-points}%
\end{figure}

The following lemma ensures that, under appropriate conditions for the
derivative of the local map $f_{0}$, if we choose a curve $v$ which links the
two lines $\left\{  y=-Lx\right\}  $ and $\left\{  y=Lx\right\}  $ then under
a sufficiently high iterate of the map we will have $f_{0}^{n}(
\tilde{S}_{0}^{\pm}(\delta))  \cap v\neq\emptyset$, see Figure
\ref{fig:exit-points}, which means that we can reach the curve $\tilde v$ from the set $\tilde{S}_{0}^{\pm}(\delta)$. The lemma ensures also a mirror result for $f_{1}$.
More formally we state it as:

\begin{lemma}
\label{lem:exit-points}Let $\tilde{U}\subset\mathbb{R}^{2}$ be a cartesian
product of two closed intervals $\left[  -\alpha,\alpha\right]  \times\left[
-\beta,\beta\right]  $ with $\beta/\alpha>L$. Let us consider a fixed
$i\in\left\{  0,1\right\}  $. Let $\tilde{v}:\left[  -1,1\right]
\rightarrow\tilde{U}$ be a continuous function satisfying $\tilde{v}%
(-1)\in\left\{  y=-Lx\right\}  $, $\tilde{v}(1)\in\left\{  y=Lx\right\}  $ and
$\left\vert \pi_{x}\tilde{v}\left(  \left[  -1,1\right]  \right)  \right\vert
>0$. If for every $C\in\lbrack Df_{i}(\tilde{U})]$ we have
\begin{equation}
\left\vert y\right\vert \leq L\qquad\implies\qquad\left\vert \pi_{x}C\left(
1,y\right)  \right\vert >1\quad\text{and\quad}\left\vert \pi_{y}C\left(
1,y\right)  \right\vert \leq L,\label{eq:segment-prop-key-assumption}%
\end{equation}
then for every $\delta>0$  there exist a point $p\in\tilde
{S}_{i}^{\pm}\left(  \delta\right)  $ and an integer $n\geq0$ such that
\[
f_{i}^{k}\left(  p\right)  \in\tilde{U}\qquad\text{for }0\leq k\leq n,
\]
and
\[
f_{i}^{n}\left(  p\right)  \in\tilde{v}\left(  \left[  -1,1\right]  \right)  .
\]

\end{lemma}

\begin{proof}
Without the loss of generality let us assume that $\pi_{x}\tilde{v}\left(  \left[
-1,1\right]  \right)  >0$.  See Figure
\ref{fig:exit-points}.

The $[  Df_{i}(  \tilde{U}) ]  $ and $\left[
-L,L\right]  $ are compact, so there exists a $\lambda>1$ such that if
$\left\vert y\right\vert \leq L$ and  $C\in[  Df_{i}(  \tilde
{U})  ]  $ then
\[
\left\vert \pi_{x}C\left(  1,y\right)  \right\vert >\lambda.
\]

For every $q_{1},q_{2}\in\tilde{U}$
\begin{align*}
f_{i}\left(  q_{1}\right)  -f_{i}\left(  q_{2}\right)   &  =\int_{0}^{1}%
\frac{d}{ds}f_{i}\left(  q_{2}+s\left(  q_{1}-q_{2}\right)  \right)  ds\\
&  =\int_{0}^{1}Df_{i}\left(  q_{2}+s\left(  q_{1}-q_{2}\right)  \right)
ds\,\left(  q_{1}-q_{2}\right)  =C\left(  q_{1}-q_{2}\right)
\end{align*}
for%
\[
C=C\left(  q_{1},q_{2}\right)  =\int_{0}^{1}Df_{i}\left(  q_{2}+s\left(
q_{1}-q_{2}\right)  \right)  ds\in[ Df( \tilde{U}) ] .
\]
This has a number of consequences under our assumption
(\ref{eq:segment-prop-key-assumption}). For two points $q_{1},q_{2}$ such that
$\pi_{x}\left(  q_{1}-q_{2}\right)  \neq0$ and $\left\vert \pi_{y}\left(
q_{1}-q_{2}\right)  \right\vert \leq L\left\vert \pi_{x}\left(  q_{1}%
-q_{2}\right)  \right\vert $ we will have
\begin{align*}
\left\vert \pi_{x}\left(  f_{i}\left(  q_{1}\right)  -f_{i}\left(
q_{2}\right)  \right)  \right\vert  &  =\left\vert \pi_{x}C\left(  q_{1}%
-q_{2}\right)  \right\vert \\
&  =\left\vert \pi_{x}\left(  q_{1}-q_{2}\right)  \right\vert \left\vert
\pi_{x}C\frac{q_{1}-q_{2}}{\left\vert \pi_{x}\left(  q_{1}-q_{2}\right)
\right\vert }\right\vert >\lambda\left\vert \pi_{x}\left(  q_{1}-q_{2}\right)
\right\vert .
\end{align*}
Moreover,
\begin{align*}
\left\vert \pi_{y}\left(  f_{i}\left(  q_{1}\right)  -f_{i}\left(
q_{2}\right)  \right)  \right\vert  &  =\left\vert \pi_{y}C\left(  q_{1}%
-q_{2}\right)  \right\vert \\
&  =\left\vert \pi_{x}\left(  q_{1}-q_{2}\right)  \right\vert \left\vert
\pi_{y}C\frac{q_{1}-q_{2}}{\left\vert \pi_{x}\left(  q_{1}-q_{2}\right)
\right\vert }\right\vert \leq\left\vert \pi_{x}\left(  q_{1}-q_{2}\right)
\right\vert L,
\end{align*}
hence%
\[
\left\vert \pi_{y}\left(  f_{i}\left(  q_{1}\right)  -f_{i}\left(
q_{2}\right)  \right)  \right\vert \leq L\left\vert \pi_{x}\left(
f_{i}\left(  q_{1}\right)  -f_{i}\left(  q_{2}\right)  \right)  \right\vert .
\]
This means that starting from two points $q_{1},q_{2}$ such that
$\pi_{x}\left(  q_{1}-q_{2}\right)  \neq0$ and $\left\vert \pi_{y}\left(
q_{1}-q_{2}\right)  \right\vert \leq L\left\vert \pi_{x}\left(  q_{1}%
-q_{2}\right)  \right\vert $ we can iterate the argument as long as $f_{i}^{m-1}(
q_{1}),$ $f_{i}^{m-1}(  q_{2})  \in\tilde{U}$ to obtain%
\begin{align}
\left\vert \pi_{x}\left(  f_{i}^{m}\left(  q_{1}\right)  -f_{i}^{m}\left(
q_{2}\right)  \right)  \right\vert  &  >\lambda^{m}\left\vert \pi_{x}\left(
q_{1}-q_{2}\right)  \right\vert ,\label{eq:x-expansion}\\
\left\vert \pi_{y}\left(  f_{i}^{m}\left(  q_{1}\right)  -f_{i}^{m}\left(
q_{2}\right)  \right)  \right\vert  &  \leq L\left\vert \pi_{x}\left(
f_{i}^{m}\left(  q_{1}\right)  -f_{i}^{m}\left(  q_{2}\right)  \right)
\right\vert .\label{eq:y-alignment}
\end{align}

By (\ref{eq:x-expansion}) the function $z\mapsto\pi_{x}f_{i}^{m}\left(
z,0\right)  $ is injective. A conclusion which follows from
(\ref{eq:x-expansion}--\ref{eq:y-alignment}) is that for as long as the curve $f_{i}^{m}\left(
\left[  -\delta,\delta\right]  ,0\right)  $ does not exit $U$ it is a graph of
a Lipschitz function, with the Lipschitz constant equal to $L$.

Due to (\ref{eq:x-expansion}) and since $f_{i}\left(  0\right)  =0$, for some
$\delta_{n}\in\lbrack-\delta,\delta]$ and for sufficiently large $n$ we will
have%
\[
\pi_{x}f_{i}^{n}\left(  \delta_{n},0\right)  >\pi_{x}\tilde{U}\qquad
\text{and}\qquad f_{i}^{n-1}\left(  \delta_{n},0\right)  \in\tilde{U}.
\]
This means that the curve $\left[  0,\delta_{n}\right]  \ni z\mapsto\pi
_{x}f_{i}^{n}\left(  z,0\right)  \in\mathbb{R}^{2}$ (or the curve $\left[
\delta_{n},0\right]  \mapsto\pi_{x}f_{i}^{n}\left(  z,0\right)  $ in the case
that $\delta_{n}<0$) will intersect with the curve $\left[  -1,1\right]  \ni
z\mapsto\tilde{v}\left(  z\right)  \in\tilde{U}$; see right plot in Figure
\ref{fig:exit-points}. Since $\left[  -\delta_{n},\delta_{n}\right]
\times\left\{  0\right\}  \subset\tilde{S}_{i}^{\pm}\left(  \delta\right)  $
this means that%
\[
f_{i}^{n}(\tilde{S}_{i}^{\pm}\left(  \delta\right)  )\cap\tilde{v}\left(
\left[  -1,1\right]  \right)  \neq\emptyset,
\]
which concludes our proof.
\end{proof}

\begin{figure}[ptb]
\begin{center}
\includegraphics[width=9cm]{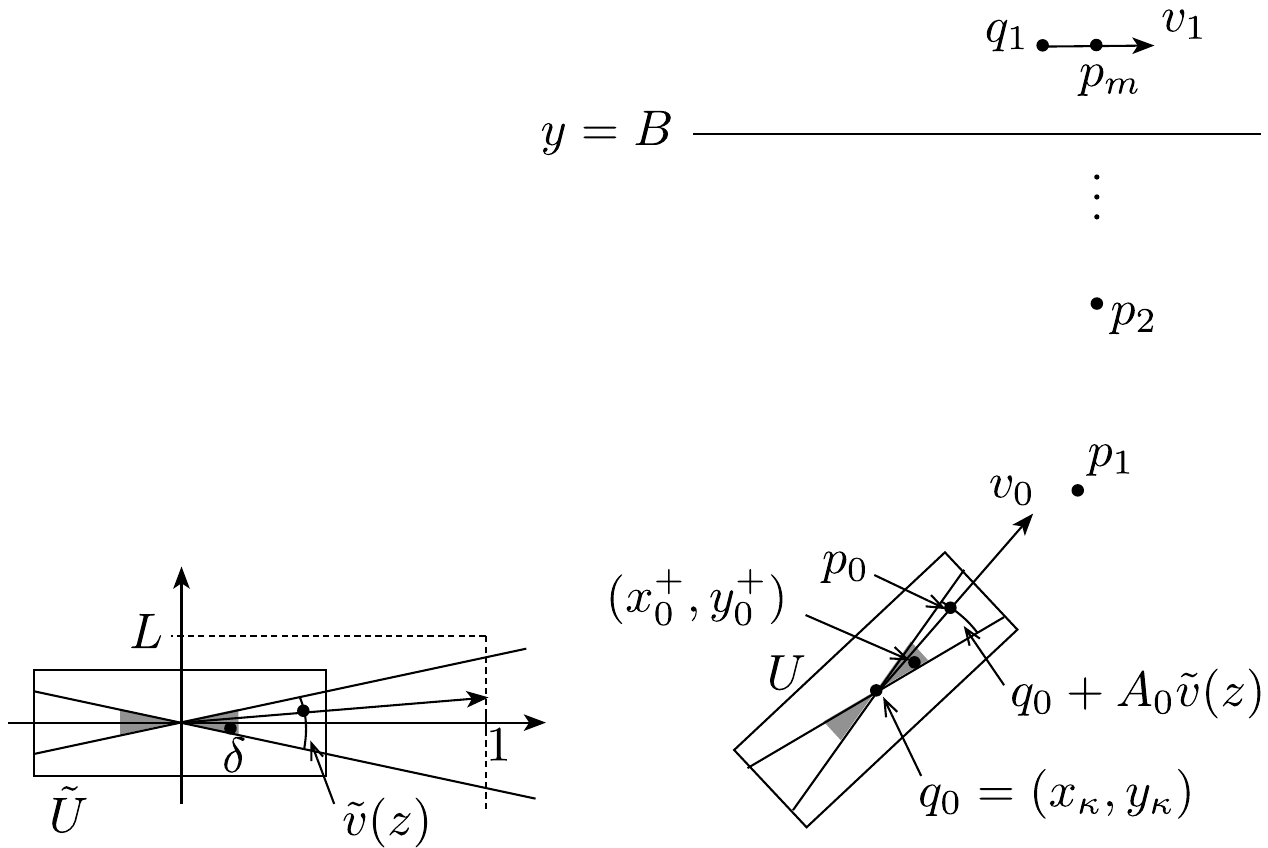}
\end{center}
\caption{Establishing a point which goes up from a segment at $(x_{\kappa
},y_{\kappa})$. }%
\label{fig:going-up}%
\end{figure}
\begin{remark} A good choice of the matrices $A_0$ and $A_1$ to ensure (\ref{eq:segment-prop-key-assumption}) is for their first columns to consist of the unstable eigenvectors of the Jacobian of $f_{a,b}$ at the fixed points $(x_\kappa,y_\kappa)$ and  $(x_{-\kappa},y_{-\kappa})$, respectively, and for their second columns to consist of the stable eigenvectors.
\end{remark}

\subsubsection{Going above and below a line from a segment}
We now show that the parallel shooting method involving the Krawczyk theorem from section \ref{sec:parallel-shooting} can be
combined with Lemmas \ref{lem:small-delta} and \ref{lem:exit-points} to ensire
conditions (I)--(VIII). As in section \ref{sec:parallel-shooting} we will
consider fixed vectors $q_{0},q_{1},v_{0},v_{1}\in\mathbb{R}^{2}$ and
find $h_{0},h_{1}$ and  $\{p_{i}\}_{i=0}^{m}$ such that we obtain a trajectory
satisfying (\ref{eq:trajectory}). We set up our problem similarly to what is
done in section \ref{sec:parallel-shooting}, but incorporate a parameter into
our discussion. We consider
\[
F:\left[  -1,1\right]  \times\mathbb{R}\times\mathbb{R}\times\underset{m-1}{\underbrace{\mathbb{R}%
^{2}\times\ldots\times\mathbb{R}^{2}}}=\left[  -1,1\right]  \times\mathbb{R}%
^{2m}\rightarrow\mathbb{R}^{2m}
%F:\left[  -1,1\right]  \times\mathbb{R}^{2}\times\mathbb{R}^{2\left(
%m-1\right)  }=\mathbb{R}^{2m}\rightarrow\mathbb{R}^{2m}%
\]
defined as%
\begin{equation}
F\left(  z,h_{0},h_{1},p_{1},\ldots,p_{m-1}\right)  =\left(
\begin{array}
[c]{l}%
f\left(  p_{1}\right)  -p_{2}\\
f(p_{2})-p_{3}\\
\qquad\vdots\\
f\left(  p_{m-2}\right)  -p_{m-1}\\
f\left(  p_{m-1}\right)  -\left(  q_{1}+h_{1}v_{1}\right)  \\
f\left(  q_{0}+h_{0}v_{0}\left(  z\right)  \right)  -p_{1}%
\end{array}
\right)  .\label{eq:shooting-parameter}%
\end{equation}
We emphasise that in (\ref{eq:shooting-parameter}) we allow the vector $v_{0}\left(  z\right)  $ to be parameter
dependent, but we keep $v_1$ fixed. We will use the following version of the Krawczyk method for parameter
dependent maps:

\begin{thm}
\label{th:krawczyk-parameter}\cite{Alfred}Let $\left[  X\right]
\subset\mathbb{R}^{n}$ be an interval set, let $\left[  Z\right]
\subset\mathbb{R}$ be a closed interval, let $X_{0}\in\left[  X\right]  $ and
let $C\in\mathbb{R}^{n\times n}$ be a linear isomorphism. Let $F:\left[
Z\right]  \times\mathbb{R}^{n}\rightarrow\mathbb{R}^{n}$ be a $C^{1}$ map and
let $\left[  D_{X}F\left(  \left[  Z\right]  ,\left[  X\right]  \right)
\right]  \subset\mathbb{R}^{n\times n}$ stand for the interval enclosure of
the map%
\begin{align*}
& \left.  \left[  D_{X}F\left(  \left[  Z\right]  ,\left[  X\right]  \right)
\right]  :=\right.  \\
& \left\{  \left(  a_{ij}\right)  _{i,j=1,...,n}:a_{ij}\in\lbrack\inf
_{z\in\left[  Z\right]  ,p\in\left[  X\right]  }\frac{\partial F_{i}}{\partial
x_{j}}\left(  z,p\right)  ,\sup_{z\in\left[  Z\right]  ,p\in\left[  X\right]
}\frac{\partial F_{i}}{\partial x_{j}}\left(  z,p\right)  ]\right\}  .
\end{align*}
Let
\[
K:=X_{0}-CF\left(  X_{0}\right)  +\left(  Id-C\left[  D_{X}F\left(  \left[
Z\right]  ,\left[  X\right]  \right)  \right]  \right)  \left(  \left[
X\right]  -X_{0}\right)  .
\]
If $K\subset\mathrm{int}\left[  X\right]  $ then for every $z\in\left[
Z\right]  $ there exists exactly one $x\left(  z\right)  \in$ $\left[
X\right]  $ for which $F\left(  z,x(z)\right)  =0.$ Moreover, $z\mapsto
x\left(  z\right)  $ is $C^{1}$. 
\end{thm}
If by means of Theorem \ref{th:krawczyk-parameter} we validate that $F=0$, and if we check that $q_1 + h_1(z) v_1$ lies above $B$ (or below $-B$), then we obtain a trajectory starting from $q_0 + h_0(z) v_0(z)$, which goes above $B$ (or below $-B$), for every $z\in [Z]$.

\subsubsection{The validation procedure\label{sec:validation-dsf}}
We now demonstrate how Theorem \ref{th:krawczyk-parameter} applied to
(\ref{eq:shooting-parameter}) and combinded with Lemmas \ref{lem:small-delta}
and \ref{lem:exit-points} can be used to ensure conditions (I)--(VIII). We do
this on an example how to validate that a point $\left(  x_{0}^{+},y_{0}%
^{+}\right)  $ goes above the line $y=B$; namely we discuss below how to establish (I) and (II). The illustration for the method described below can be found in
Figure \ref{fig:going-up}. 

First, a good choice for us in this setting will be
\[
q_{0}=\left(  x_{\kappa},y_{\kappa}\right)  .
\]
We will also choose $v_{0}(z):=A_{0}\left(  1,zL\right)  $, for $z\in
\lbrack-1,1]$. With the use of Theorem \ref{th:krawczyk-parameter} applied to
(\ref{eq:shooting-parameter}) we establish the existence and bounds for
$h_{0}\left(  z\right)  ,h_{1}(z)\in\mathbb{R}$ and $p_{0}\left(  z\right)
,\ldots,p_{m}\left(  z\right)  $ for which%
\begin{align}
p_{0}\left(  z\right)   &  =q_{0}+h_{0}\left(  z\right)  v_{0}\left(
z\right)  , \notag\\
p_{k}\left(  z\right)   &  =f\left(  p_{k-1}\left(  z\right)  \right)
\qquad\text{for }k=1,\ldots,m, \label{eq:zhooting-z} \\
p_{m}\left(  z\right)   &  =q_{1}+h_{1}(z)v_{1}. \notag
\end{align}
%We apply Theorem \ref{th:krawczyk-parameter} to obtain the existence and to validate the bounds for $h_0(z),h_1(z),p_0(z),\ldots, p_m(z)$. 
After doing this we need to check that from the obtained bounds it follows that
$p_{m}(z)$ lies above $y=B$. We also need to ensure that for $\tilde
{v}(z):=h_{0}\left(  z\right) \left(  1,zL\right)  $ we have $ \tilde{v}([-1,1])\subset \tilde{U}$, that $|h_0([-1,1])|>0$, and validate assumptions of Lemma \ref{lem:exit-points} on
the set $\tilde{U}$. This ensures that we can reach some point $p_{0}\left(  z\right)
=q_{0}+A_{0}\tilde{v}(z)$ from $S_{0}^{+}\left(  \delta\right)  $ after $n$
iterates. This in particular implies that there exists a point $(x_{0}^{+},y_{0}^{+})$ in the
segment $S_{0}^{+}\left(  \delta\right)  $ whose $n$-th iterate arrives at
$p_{0}\left(  z\right)  =q_{0}+A_{0}\tilde{v}(z)$. We know from (\ref{eq:zhooting-z}) that from $p_{0}\left(
z\right)  $ we can go to $p_m(z)$, which is above the line $y=B$, which
implies the condition (II). It is important to note that Lemma
\ref{lem:exit-points} ensures that we can choose arbitrarily small $\delta$.
Thus, we can go above the line $y=B$ from an arbitrarily small segment
$S_{0}^{\pm}\left(  \delta\right)  $. Lemma \ref{lem:small-delta} ensures (I)
by choosing sufficiently small $\delta$.

Above description gives us a method which we can use to establish that:

\begin{enumerate}
\item[(I)-(II)] There exists a segment $S_{0}^{+}\left(  \delta\right)  $ that
starts at $\left(  x_{\kappa},y_{\kappa}\right)  $ and ends at a point
$(x_{0}^{+},y_{0}^{+})$ which goes above the line $y=B$ after some number of
steps. Moreover, $f_{a,b}\left(  S_{0}^{+}\left(  \delta\right)  \right)
\subset B_{\ell}^{\kappa}$ for $\ell=1,...,N$.
\end{enumerate}
Mirror arguments can be used to show that we have (III)-(IV), (V)-(VI) and (VII)-(VIII).

\subsection{Results}
To calculate the chaotic area we proceed as follows.
We subdivide the $(a,b)$ parameter range $[3,10]\times[0.1,0.8]$ into $700\times 100$ initial boxes. On each initial box we attempt a validation, and include that box towards the chaotic are if it succeeds. If a validation fails, we subdivide the box into $5\times 5$ smaller boxes, and on each we attempt the validation again. The subdivision process is repeated three times, so the (theoretical) number of smallest boxes inside each initial box could be $25^3$. In practice our validation succeeds well before the boxes are subdivided into such small fragments. In some parameter regions though, even on the smallest parameter boxes the validation can fail, in wich case the box is not counted towards the chaotic area. 

In more detail, for each parameter box we proceed as follows, repeating the procedure three times:
\begin{enumerate}
\item[(step 1)] Starting with $\kappa=0$ and finishing with the largest integer $\kappa$ which is smaller than $a(1-b)^{-1}$, we apply the procedure from section \ref{sec:validation-dsf} to validate: (I)-(II), (III)-(IV), (V)-(VI) and (VII)-(VIII). If for some $\kappa$ we succeed, then we can stop and include the parameter box as the validated chaotic area.
\item[(step 2)] If for each possible $\kappa$ we are not able to validate (I)-(VIII), then we subdivide the box into $25$ equal sub-boxes, and for each of them go back to (step 1). 
\end{enumerate}
Using the above procedure we have validated that the area of the parameter domain $[3,10]\times[0.1,0.8]$ for which we have chaos exceeds $98.47\%$ of the domain, proving Theorem \ref{thm:cap-dsf}.

The computer-assisted proof has been conducted on a cluster, running on 48 parallel threads under two hours. (The computation would have taken $2\cdot 48$ hours on a single thread.) The code used for the proof is available in \cite{code}.

\bibliographystyle{plain}
\bibliography{bibliografia2}

\end{document}